\documentclass[onefignum,onetabnum]{siamonline190516}

\usepackage{lipsum}
\usepackage{amsfonts}
\usepackage{graphicx}
\usepackage{epstopdf}
\usepackage{algorithmic}

\usepackage{mathrsfs}
\usepackage{amsmath,amssymb}
\usepackage{bbm}

\ifpdf
\DeclareGraphicsExtensions{.eps,.pdf,.png,.jpg}
\else
\DeclareGraphicsExtensions{.eps}
\fi

\usepackage{enumitem}
\setlist[enumerate]{leftmargin=.5in}
\setlist[itemize]{leftmargin=.5in}

\newtheorem{thm}{Theorem}
\newtheorem{defin}[thm]{Definition}     
\newtheorem{lem}[thm]{Lemma}
\newtheorem{cor}[thm]{Corollary}
\newtheorem{ex}[thm]{Example}
\newtheorem{rem}[thm]{Remark}

\newcommand{\eps}{\varepsilon}
\newcommand{\E}{\mathbb{E}}
\renewcommand{\P}{\mathbb{P}}
\renewcommand{\H}{\mathbb{H}}
\newcommand{\R}{\mathbb{R}}

\newcommand{\W}{\mathbb{W}}

\newcommand{\N}{\mathbb{N}}

\newsiamremark{remark}{Remark}
\newsiamremark{hypothesis}{Hypothesis}
\crefname{hypothesis}{Hypothesis}{Hypotheses}
\newsiamthm{claim}{Claim}

\headers{PDE-constrained regression}{R. Nickl, S. van de Geer and S. Wang}

\title{Convergence rates for Penalised Least Squares estimators in PDE-constrained regression problems\thanks{Submitted to the editors 30/12/2018.
\funding{RN and SW were supported by ERC grant No. 647812, SW by EPSRC grant EP/L016516/1 and CCIMI, and SvdG by EPSRC Grant LNAG/036 RG91310 during her visit to the Isaac Newton Institute, Cambridge (Jan-Jun 2018), when this research was initiated.}}}

\author{Richard Nickl\thanks{Department of Pure Mathematics \& Mathematical Statistics, Univ. of Cambridge
  (\email{r.nickl@statslab.cam.ac.uk}).}
\and Sara van de Geer\thanks{Seminar f\"ur Statistik, ETH Z\"urich
  (\email{sara.vandegeer@stat.math.ethz.ch}).}
\and Sven Wang\thanks{Department of Pure Mathematics \& Mathematical Statistics, Univ. of Cambridge (\email{ssw29@cam.ac.uk}).}}

\usepackage{amsopn}

\begin{document}

\maketitle

% REQUIRED
\begin{abstract}
We consider PDE constrained nonparametric regression problems in which the parameter $f$ is the unknown coefficient function of a second order elliptic partial differential operator $L_f$, and the unique solution $u_f$ of  the boundary value problem
$$L_fu=g_1\textnormal{ on }\mathcal O, \quad u=g_2\textnormal{ on }\partial \mathcal O,$$
is observed corrupted by additive Gaussian white noise. Here $\mathcal O$ is a bounded domain in $\mathbb R^d$ with smooth boundary $\partial \mathcal O$, and $g_1, g_2$ are given functions defined on $\mathcal O, \partial \mathcal O$, respectively. Concrete examples include $L_fu=\Delta u-2fu$ (Schr\"odinger equation with attenuation potential $f$) and $L_fu=\text{div} (f\nabla u)$ (divergence form equation with conductivity $f$). In both cases, the parameter space
\[\mathcal F=\{f\in H^\alpha(\mathcal O)| f > 0\}, ~\alpha>0, \]
where $H^\alpha(\mathcal O)$ is the usual order $\alpha$ Sobolev space, induces a set of non-linearly constrained regression functions $\{u_f: f \in \mathcal F\}$.
%\par
We study Tikhonov-type penalised least squares estimators $\hat f$ for $f$. The penalty functionals are of squared Sobolev-norm type and thus $\hat f$ can also be interpreted as a Bayesian `MAP'-estimator corresponding to some Gaussian process prior. We derive rates of convergence of $\hat f$ and of $u_{\hat f}$, to $f, u_f$, respectively. We prove that the rates obtained are minimax-optimal in prediction loss.  Our bounds are derived from a general convergence rate result for non-linear inverse problems whose forward map satisfies a modulus of continuity condition, a result of independent interest that is applicable also to linear inverse problems, illustrated in an example with the Radon transform.
\end{abstract}

% REQUIRED

% REQUIRED

	\section{Introduction}
	Observations obeying certain physical laws can often be described by a partial differential equation (PDE). Real world measurements carry statistical noise and thus do not generally exactly exhibit the idealised pattern of the PDE, but it is desirable that recovery of  parameters from data is consistent with the PDE structure. In the mathematical literature on inverse problems several algorithms that incorporate such constraints have been proposed, notably optimisation based methods such as Tikhonov regularisation \cite{EHN96, BB18} and maximum a posteriori (MAP) estimates related to Bayesian inversion techniques \cite{stuart10, DS16}. In statistical terminology these methods can be viewed as penalised least squares estimators over parameter spaces of regression functions that are restricted to lie in the range of some `forward operator' $\mathscr G$ describing the solution map of the PDE. The case where $\mathscr G$ is linear is reasonably well studied in the inverse problems literature, but already in basic elliptic PDE examples, the map $\mathscr G$ is \textit{non-linear} and the analysis is more involved. The observation scheme considered here will be a natural continuous analogue of the standard Gaussian regression model
	\begin{equation}\label{intro-obs}
	Y_i=u_f(x_i)+\varepsilon_i, ~i=1, \dots, n; ~\{\varepsilon_i\} \sim^{i.i.d.}N(0,1),
	\end{equation}
	where $(x_i)_{i=1}^n$ are `equally spaced' design points on a bounded domain $\mathcal O \subset \mathbb R^d$ with smooth boundary $\partial \mathcal O$. The function $u_f: \mathcal O \to \mathbb R$ is, in our first example, the solution $u=u_f$ of the elliptic PDE (with $\nabla$ denoting the gradient and $\nabla \cdot$ the divergence operator)
	\begin{equation}\label{intro-div}
	\begin{cases}
	\nabla\cdot (f\nabla u) =g \quad \textnormal{on }  \mathcal O,\\
	u=0\quad \textnormal{on }  \mathcal \partial O,
	\end{cases}
	\end{equation}
	where $g>0$ is a given source function defined on $\mathcal O$ and $f: \mathcal O \to (0,\infty)$ is an unknown \textit{conductivity  (or diffusion) coefficient}. The second model example arises with solutions $u=u_f$ of the time-independent Schr\"odinger equation  (with $\Delta$ equal to the standard Laplacian operator)
	\begin{equation}\label{intro-sch}
	\begin{cases}
	\Delta u -2fu =0 \quad \textnormal{on }  \mathcal O,\\
	u=g\quad \textnormal{on } \mathcal \partial O,
	\end{cases}
	\end{equation}
	corresponding to the unknown \textit{attenuation potential (or reaction coefficient)}  $f: \mathcal O \to (0,\infty)$, and given positive `boundary temperatures' $g>0$. Both PDEs have a fundamental physical interpretation and feature in many application areas, see, e.g., \cite{EHN96, BHM04, HP08, BU10, stuart10, devore, DS16}, and references therein. 
	
	When $f>0$ belongs to some Sobolev space $H^\alpha(\mathcal O)$ for appropriate $\alpha>0$, unique solutions $u_f$ of the PDEs (\ref{intro-div}), (\ref{intro-sch}) exist, and the `forward' map $f \mapsto u_f$ is non-linear. [In fact, in (\ref{intro-sch}) only $f \ge 0$ is required.] A natural method to estimate $f$ is by a penalised least squares approach: one minimises over $f \in H^\alpha(\mathcal O)$ with $f>0$ the squared Euclidean distance $$Q_n(f)=\|Y-u_f\|^2$$ of the observation vector $(Y_i:i=1, \dots, n)$ to the fitted values $(u_f(x_i):i=1, \dots, n)$, and penalises too complex solutions $f$ by, for instance, an additive Sobolev norm $\|\cdot\|_{H^\alpha}$ - type penalty. The (from a PDE perspective) natural constraint $f>0$ can be incorporated by a smooth one-to-one transformation $\Phi$ of the penalty function, and a final estimator $\hat f$ minimises a criterion function of the form $$Q_n(f) + \lambda^2 \|\Phi^{-1}[f]\|_{H^\alpha}^2,$$ over $f \in H^\alpha(\mathcal O)$ with $f>0$, where $\lambda$ is a scalar regularisation parameter to be chosen. Both Tikhonov regularisers as well as Bayesian maximum a posteriori (MAP) estimates arising from suitable Gaussian priors fall into this class of estimators. We show in the present paper that suitable choices of $\lambda, \alpha, \Phi$ give rise to statistically optimal solutions of the above PDE constrained regression problems from data (\ref{intro-obs}), in prediction loss. The convergence rates obtained can be combined with `stability estimates' to obtain bounds also for the recovery of the parameter $f$ itself. 
	
	Our main results are based on a general convergence rate theorem for minimisers over $H^\alpha$ of  functionals of the form $$F \mapsto \|Y-\mathscr G(F)\|^2 + \lambda^2 \|F\|_{H^\alpha}^2$$ in possibly non-linear inverse problems whose forward map $F \mapsto \mathscr G(F)$ satisfies a certain modulus of continuity assumption between Hilbert spaces. This result, which adapts $M$-estimation techniques \cite{sara2001, saramest} to the inverse problems setting, is of independent interest, and provides novel results also for linear forward maps, see Remark \ref{radon} for an application to Radon transforms.
	
	For sake of conciseness, our theory is given in the Gaussian white noise model introduced in (\ref{data}) below -- it serves as an asymptotically equivalent (see \cite{BL96, R08}) continuous analogue of the discrete model (\ref{intro-obs}), and facilitates the application of PDE techniques in our proofs. Transferring our results to discrete regression models is possible, but the additional difficulties are mostly of a technical nature and will note be pursued here.
	
	Recovery for non-linear inverse problems such as those mentioned above has been studied initially in the deterministic regularisation literature \cite{EKN89, N92, SEK93, EHN96, TJ02}, and the convergence rate theory developed there has been adapted to the statistical regression model (\ref{intro-obs}) in \cite{BHM04, BHMR07, HP08, LL10}. These results all assume that a suitable Fr\'echet derivative $D\mathscr G$ of the non-linear forward map $\mathscr G$ exists at the `true' parameter $F$, and moreover require that $F$ lies in the range of the adjoint operator of $D\mathscr G$ -- the so called `source condition'. Particularly for the PDE (\ref{intro-div}), such conditions are problematic and do not hold in general for rich enough classes of $F$'s (such as Sobolev balls) unless one makes very stringent additional model assumptions. Our results circumvent such source conditions. Further remarks, including a discussion of related convergence analysis of estimators obtained from Bayesian inversion techniques \cite{v13, dashti13, n17} can be found in Section \ref{sec-pde-rem}.
	
	The article is organised as follows. The main results are stated in Sections \ref{sec-general} and \ref{sec-pderes}; their proofs are contained in Sections \ref{sec-pfs} and \ref{sec-ex-pf}. Some key auxiliary results about the elliptic PDE (\ref{intro-div})-(\ref{intro-sch}) and the `link functions' $\Phi$ used below are proved in Section \ref{sec-pde-facts} and \ref{sec-reg} respectively.

	\subsection{Some preliminaries and basic notation}\label{sec-gen-pre}
	Throughout, $\mathcal O\subseteq \mathbb R^d$, $d\geq 1$, denotes a bounded non-empty $C^\infty$-domain (an open bounded set with smooth boundary) with closure $\bar{\mathcal O}$. The usual space $L^2(\mathcal O)$ of square integrable functions carries a norm $\|\cdot\|_{L^2(\mathcal O)}$ induced by the inner product $$\langle h_1, h_2 \rangle_{L^2(\mathcal O)} = \int_\mathcal O h_1(x)h_2(x)dx,~~h_1,h_2 \in L^2(\mathcal O),$$ where $dx$ denotes Lebesgue measure. For any multi-index $i=(i_1,...,i_d)$ of `order' $|i|$, let $D^i$ denote the $i$-th (weak) partial derivative operator of order $|i|$. Then for integer $\alpha\geq 0$, the usual Sobolev spaces are defined as
	\[H^\alpha(\mathcal O):=\left\{f\in L^2(\mathcal O)\; \middle| \;\textnormal{for all }|i|\leq \alpha, \; D^if \textnormal{ exists and } D^if\in L^2(\mathcal O) \right\},\]
	normed by $\|f\|_{H^\alpha(\mathcal O)} = \sum_{|i| \le \alpha} \|D^if\|_{L^2(\mathcal O)}$. For non-integer real values $\alpha\geq 0$, we define $H^\alpha(\mathcal O)$ by interpolation, see, e.g., \cite{lionsmagenes} or \cite{T78}. 
	
	\par
	
	The spaces of bounded and continuous functions on $\mathcal O$ and $\bar{\mathcal O}$ are denoted by $C(\mathcal O)$ and $C(\bar{\mathcal O})$, respectively, equipped with the supremum norm $\|\cdot\|_\infty$. For $\eta\in\N$, the space of $\eta$-times  differentiable functions on $\mathcal O$ with (bounded) uniformly continuous derivatives is denoted by $C^\eta(\mathcal O)$. For $\eta > 0,\eta\notin \N$, we say $f \in C^\eta(\mathcal O)$ if for all multi-indices $\beta$ with $|\beta|\leq \lfloor\eta\rfloor $ (the integer part of $\eta$), $D^\beta f$ exists and is $\eta-\lfloor\eta\rfloor$-H\"older continuous. The norm on $C^\eta(\mathcal O)$ is 
	\[\|f\|_{C^\eta(\mathcal O)}=\sum_{\beta:|\beta|\leq\lfloor\eta\rfloor}\|D^\beta f\|_{\infty}+\sum_{\beta:|\beta|=\lfloor\eta\rfloor}\sup_{x,y\in\mathcal O, \;x\neq y}\frac{|D^\beta f(x)-D^\beta f(y)|}{|x-y|^{\eta-\lfloor \eta\rfloor}}. \] We also define the set of smooth functions as $C^\infty(\mathcal O)=\cap_{\eta>0} C^\eta(\mathcal O)$ and its subspace $C^\infty_c(\mathcal O)$ of functions compactly supported in $\mathcal O$. 
	
	The previous definitions will be used also for $\mathcal O$ replaced by $\partial \mathcal O$ or $\mathbb R^d$. When there is no ambiguity, we omit $\mathcal O$ from the notation.
	\par 
	For any normed linear space $(X,\|\cdot\|_X)$ its topological dual space is 
	\[X^*:=\left\{L:X\to \mathbb R \textnormal{ linear s.t. } \exists C>0 \;\forall x\in X: \;|L(x)|\leq C\|x\|_X \right\},\]
	which is a Banach space for the norm $\|L\|_{X^*}=\sup_{x\in X}|L(x)|/ \|x\|_{X}.$ 
	
	We need further Sobolev-type spaces to address routine subtleties of the behaviour of functions near $\partial \mathcal O$: denote by $H^\alpha_c(\mathcal O)$ the completion of $C^\infty_c(\mathcal O)$ for the $H^\alpha(\mathcal O)$-norm, and let $\tilde H^\alpha(\mathcal O)$ denote the closed subspace of $H^\alpha(\mathbb R^d)$ consisting of functions supported in $\bar {\mathcal O}$. We have $H^\alpha_c(\mathcal O)=\tilde H^\alpha(\mathcal O)$ unless $\alpha=k+1/2, k \in \mathbb N$ (Section 4.3.2 in \cite{T78}), and one defines negative order Sobolev spaces $H^{-\kappa}(\mathcal O)=(\tilde H^\kappa(\mathcal O))^*, \kappa>0$, cf.~also Theorem 3.30 in \cite{ML00}.
	
	We use the symbols ``$\lesssim, \gtrsim$'' for inequalities that hold up to multiplicative constants that are universal, or whose dependence on other constants will be clear from the context. We also use the standard notation $\R_+:= \{x|x\geq 0\}$ and $a\vee b:=\max\{a,b\}$ for $a,b\in \R$.

	\section{A convergence rate result for general inverse problems}\label{sec-general}

	\subsection{Forward map and white noise model}
	
	Let $\mathbb H$ be a separable Hilbert space with inner product $\langle \cdot, \cdot \rangle_\mathbb H$. Suppose that $\tilde{\mathcal V}\subseteq L^2(\mathcal O)$ and that 
	\begin{equation*}
	\mathscr G:\tilde{\mathcal V}\to \mathbb H, \qquad F\mapsto \mathscr G(F),
	\end{equation*}
	is a given `forward' map.  For some $F\in \tilde {\mathcal  V}$, and for scalar `noise level' $\varepsilon>0$, we observe a realisation of the equation
	\begin{equation}\label{data}
	Y^{(\varepsilon)}=\mathscr G(F)+\varepsilon \mathbb W,
	\end{equation}
	where $(\mathbb W(\psi):\psi\in \mathbb H)$ is a centred Gaussian white noise process indexed by the Hilbert space $\mathbb H$ (see p.19-20 in \cite{nicklgine}). Let $\mathbb E_F^{\varepsilon}, F \in \tilde{\mathcal V},$ denote the expectation operator under the law $\mathbb P_F^\varepsilon$ of $Y^{(\varepsilon)}$ from (\ref{data}). Observing (\ref{data}) means to observe a realisation of the Gaussian process $(\langle Y^{(\eps)},\psi\rangle_{\mathbb H}: \psi \in \mathbb H)$ with marginal distributions $$\langle Y^{(\eps)},\psi\rangle_{\mathbb H} \sim N(\langle \mathscr G(F), \psi \rangle_{\mathbb H}, \eps^2\|\psi\|_{\mathbb H}^2).$$
	In the case $\mathbb H=L^2(\mathcal O)$ relevant in Section \ref{sec-pderes} below, (\ref{data}) can be interpreted as a Gaussian shift experiment in the Sobolev space $H^{-\kappa}(\mathcal O), \kappa>d/2$ (see, e.g., \cite{cn1, n17}), and also serves as a theoretically convenient (and, for $\varepsilon = 1/\sqrt n$, as $n \to \infty$ asymptotically equivalent) continuous surrogate model for observing $(Y_i, x_i)_{i=1}^n$ in the standard fixed design Gaussian regression model 
	\begin{equation}\label{disdat}
	Y_i = \mathscr G(F)(x_i) + \varepsilon_i, ~i=1, \dots, n, ~\{\varepsilon_i\} \sim^{i.i.d.} N(0,1),
	\end{equation}
	where the $x_i$ are `equally spaced' design points in the domain $\mathcal O$ (see \cite{BL96, R08}).
	
	In the discrete model (\ref{disdat}) the least squares criterion can be decomposed as 
	$\|Y-\mathscr G(F)\|_{\mathbb R^n}^2$ $=\|Y\|_{\mathbb R^n}^2 - 2\langle Y, \mathscr G(F) \rangle_{\mathbb R^n} + \|\mathscr G(F)\|_{\mathbb R^n}^2$. The first term $\|Y\|_{\mathbb R^n}^2$ is independent of $F$ and can be neglected when optimising in $F$. In the continuous model (\ref{data}) we have $\|Y\|_{\mathbb H}=\infty$ a.s. (unless dim$(\mathbb H) <\infty$), which motivates to define a `Tikhonov-regularised' functional
	
	\begin{equation}\label{Jdef2}
	\mathscr J_{\lambda, \varepsilon}:\tilde{\mathcal V}\to \mathbb R,\quad  \mathscr J_{\lambda, \varepsilon}(F):=2\langle Y^{(\varepsilon)},\mathscr G(F)\rangle_{\mathbb H}-\|\mathscr G(F)\|_{\mathbb H}^2-\lambda^2\|F\|_{H^\alpha}^2,
	\end{equation}
	where $\lambda>0$ is a regularisation parameter to be chosen, and where we set $\mathscr J_{\lambda, \varepsilon}(F)=-\infty$ for $F\notin H^\alpha$. Maximising $\mathscr J_{\lambda, \varepsilon}$ thus amounts to minimising the natural least squares fit with a $H^\alpha(\mathcal O)$-penalty for $F$, and we note that it also corresponds to maximising the penalised log-likehood function arising from (\ref{data}), see, e.g., \cite{n17}, Section 7.4.  In all that follows $\|\cdot\|_{H^\alpha}$ could be replaced by any equivalent norm on $H^\alpha(\mathcal O)$.
	
	We note that when $\mathscr G$ is non-linear, computation of a global maximiser of the (then non-convex) functional $\mathscr J_{\lambda, \varepsilon}$ may be infeasible in practice. Nevertheless, the convergence rates we obtain below provide a first rigorous understanding of the statistical complexity of the PDE inference problems at hand. It is an interesting open question whether algorithms that are computable in `polynomial time' can attain the same performance guarantees. This is subject of ongoing research (see, e.g., \cite{MNP19b}) and beyond the scope of the present paper.

	\subsection{Results}\label{sec-gen-res}
	
	For $F_1\in \tilde{\mathcal V} \cap H^\alpha$, $F_2\in\tilde{\mathcal V}$ and $\lambda >0$, define the functional
	\begin{equation}\label{taudef}
	\tau_\lambda^2(F_1,F_2):=\|\mathscr G(F_1)-\mathscr G(F_2)\|_{\mathbb H}^2+\lambda^2\|F_1\|_{H^\alpha}^2.
	\end{equation}
	The main result of this section, Theorem \ref{thm-gen}, proves the existence of maximisers $\hat F$ for $\mathscr J_{\lambda, \varepsilon}$ over suitable subsets $\mathcal V\subseteq \tilde{\mathcal V} \cap H^\alpha$ and concentration properties for $\tau_\lambda(\hat F, F_0)$, where $F_0$ is the `true' function generating the law $\mathbb P_{F_0}^\varepsilon$ from equation (\ref{data}). Note that bounds for $\tau_\lambda(\hat F, F_0)$ simultaneously control the `prediction error' $\|\mathscr G(\hat F)-\mathscr G(F_0)\|_{\mathbb H}$ as well as the regularity $\|\hat F\|_{H^\alpha}$ of the estimated output $\hat F$.  
	
	\smallskip
	
	Theorem \ref{thm-gen} is proved under a general `modulus of continuity' condition on the map $\mathscr G$ which reads as follows. 
	\begin{defin}
		Let $\alpha,\gamma, \kappa\in\R_{+}$ be non-negative real numbers and $\tilde {\mathcal V}\subseteq L^2(\mathcal O)$. Set $\mathcal H:=H^\alpha(\mathcal O)$ if $\kappa<1/2$, and $\mathcal H:=H^\alpha_c(\mathcal O)$ if $\kappa\geq 1/2$. A map $\mathscr G:\tilde{\mathcal V}\to \mathbb H$ is called $(\kappa,\gamma, \alpha)$\emph{-regular} if there exists a constant $C>0$ such that for all $F,H\in \tilde{\mathcal V}\cap \mathcal H$, we have
		\begin{equation}\label{entrcond}
		\|\mathscr G(F)-\mathscr G(H)\|_{\mathbb H}\leq C\big(1+\|F\|_{H^\alpha(\mathcal O)}^\gamma\vee \|H\|_{H^\alpha(\mathcal O)}^\gamma\big)\|F-H\|_{(H^{\kappa}(\mathcal O))^*},
		\end{equation}
	\end{defin}
	
	This condition is easily checked for `$\kappa$-smoothing' \textit{linear} maps $\mathscr G$ with $\gamma=0$, see Remark \ref{radon} for an example. But (\ref{entrcond}) also allows for certain non-linearities of $\mathscr G$ on unbounded parameter spaces $\tilde {\mathcal V}$ that will be seen later on to accommodate the forward maps induced by the PDEs (\ref{intro-div}), (\ref{intro-sch}). See also Remarks \ref{nonlin}, \ref{maphack} below.

	\begin{thm}\label{thm-gen}
		Suppose that $\mathscr G: \tilde{\mathcal V}\to \mathbb H$ is a $(\kappa,\gamma, \alpha)$-regular map for some integer $\alpha>(d/2-\kappa) \vee (\gamma d/2-\kappa)$. Let $Y^{(\eps)} \sim \mathbb P_{F_0}^\varepsilon$ from (\ref{data}) for some fixed $F_0 \in \tilde{\mathcal V}$. %Moreover, assume that $\mathcal V, \kappa$ are such that if $\kappa<1/2$, $\mathcal V\cap H^\alpha(\mathcal O)$ is closed for the weak topology of $H^\alpha(\mathcal O)$, or if $\kappa\geq 1/2$, $\mathcal V\cap H^\alpha_c(\mathcal O)$ is closed for the weak topology of $H^\alpha_c(\mathcal O)$.
		Then the following holds.
		\par 
		1. Let $\mathcal V\subseteq \tilde{\mathcal V}\cap \mathcal H$ be closed for the weak topology of the Hilbert space $\mathcal H$. Then for all $\lambda, \varepsilon>0$, almost surely under $\mathbb P_{F_0}^\varepsilon$, there exists a maximiser $\hat F=\hat F_{\lambda,\varepsilon} \in \mathcal V$ of $\mathscr J_{\lambda,\eps}$ from (\ref{Jdef2}) over $\mathcal V$, satisfying
		\begin{equation}\label{hatF}
		\sup_{F\in \mathcal V}\mathscr J_{\lambda,\varepsilon}(F) = \mathscr J_{\lambda, \varepsilon} (\hat F).
		\end{equation}
		\par 
		2. Let $\mathcal V\subseteq \tilde{\mathcal V}\cap \mathcal H$. There exist constants $c_1,c_2,c_3>0$ such that for all
		$\varepsilon,\lambda,\delta>0$ satisfying 
		\begin{equation}\label{delta-cond}
		\varepsilon^{-1}\delta \geq c_1\big(1+\lambda^{-\frac{1}{2s}}\big(1+(\delta/\lambda)^{\frac{\gamma}{2s}}\big) \big),~s:=(\alpha+\kappa)/d,
		\end{equation}
		all $R\geq \delta$, any maximiser $\hat F =\hat F_{\lambda, \varepsilon} \in \mathcal V$ of $\mathscr J_{\lambda,\eps}$ over $\mathcal V$ and any $F_* \in \mathcal V$, we have
		\begin{equation}\label{main-thm-tau}
		\mathbb P_{F_0}^\varepsilon\big(\tau_\lambda^2(\hat F,F_0)\geq 2(\tau_\lambda^2(F_*, F_0)+R^2)\big)\leq c_2\exp\Big(-\frac{R^2}{c_2^2\varepsilon^2}\Big),
		\end{equation}
		and also
		\begin{equation}\label{main-thm-rate}
		\E^\varepsilon_{F_0}\left[\tau_\lambda^2(\hat F,F_0)\right]\leq c_3 \left(\tau_\lambda^2(F_*, F_0)+\delta^2+ \varepsilon^2\right).
		\end{equation}
	\end{thm}

	Various applications of Theorem \ref{thm-gen} for specific choices of $\kappa$, $\gamma$, $\mathcal V$ and $\tilde{\mathcal V}$ will be illustrated in the following - besides the main PDE applications from Section \ref{sec-pderes}, see Remarks \ref{lin}, \ref{maphack} and \ref{bdhack} as well as Example \ref{radon} below.
	\par
	Theorem \ref{thm-gen} does not necessarily require $F_0 \in \mathcal V$ as long as $F_0$ can be suitably approximated by some $F_* \in \mathcal V$, see Remark \ref{lin} for an instance of when this is relevant. If $F_0 \in \mathcal V$ then we can set $F_*=F_0$ in the above theorem and obtain the following convergence rates, which are well known to be optimal for $\kappa$-smoothing linear forward maps $\mathscr G$, and which will be seen to be optimal also for the non-linear inverse problems arising from the PDE models (\ref{intro-div}) and (\ref{intro-sch}).
	\begin{cor}\label{cor-gen}
		Under the conditions of Part 2 of Theorem \ref{thm-gen}, for all $R>0$ there exists $c<\infty$ such that for all $\varepsilon>0$ small enough, $\lambda = \varepsilon^{2(\alpha+\kappa)/(2(\alpha+\kappa)+d)}$ and any maximizer $\hat F_{\lambda,\eps}$ of $\mathscr J_{\lambda,\eps}$ over $\mathcal V$,
		\begin{equation}\label{rate-cor}
		\sup_{F_0\in\mathcal V:\|F_0\|_{H^\alpha}\leq R}\mathbb E^{\varepsilon}_{F_0}\left\|\mathscr G(\hat F_{\lambda,\varepsilon})-\mathscr G(F_0)\right\|_{\mathbb H}\leq c\varepsilon^{\frac{2(\alpha+\kappa)}{2(\alpha+\kappa)+d}}.
		\end{equation}
	\end{cor}
	
	When images of $\|\cdot\|_{H^\alpha}$-bounded subsets of $\mathcal V$ under a forward map $\mathscr G: L^2(\mathcal O) \to L^2(\mathcal O)$ are bounded in $H^{\beta}(\mathcal O)$ for some $\beta>0$, then the $L^2$-bound (\ref{rate-cor}) extends (via interpolation and bounds for $\|\hat F\|_{H^\alpha}$ implied by Theorem \ref{thm-gen}) to $H^\eta$-norms, $\eta\in [0,\beta]$, which in turn can be used to obtain convergence rates also for $\hat F -F_0$ by using stability estimates. See the results in Section \ref{sec-pderes} and also Example \ref{radon} below for examples.

	\begin{rem}[MAP estimates] \label{lin} \normalfont Let $\Pi$ be a Gaussian process prior measure for $F$ with reproducing kernel Hilbert space (RKHS) $\mathcal H$ and RKHS-norm $\bar \lambda \|\cdot\|_{H^\alpha}, \bar \lambda>0$. Taking note of the form of the likelihood function in the model (\ref{data}) (see, e.g., Section 7.4 in \cite{n17}), maximisers $\hat F$ of $\mathscr J_{\lambda,\eps}$ over $\mathcal V=\mathcal H$ with $\lambda=\varepsilon \bar \lambda$ have a formal interpretation as maximum a posteriori (MAP) estimators for the resulting posterior distributions $\Pi(\cdot|Y^{(\varepsilon)})$, see also \cite{dashti13, HB15}. For instance, let $\alpha>d/2, \kappa\geq 0,$ and consider a \textit{linear} inverse problem where for $\beta=\alpha-d/2$ and $\tilde {\mathcal V}=H^\beta(\mathcal O)$, $\mathscr G:H^\beta(\mathcal O) \to \H$ is a linear map satisfying (\ref{entrcond}) with $\gamma=0$ for all $F,H\in H^\beta(\mathcal O)$. Then, applying Theorem \ref{thm-gen} with $\lambda= \varepsilon$ (so that $\bar \lambda =1$) and $\delta \approx \varepsilon^{(2\beta+2\kappa)/(2\beta+2\kappa+d)}$ yields
		\begin{equation}\label{maprate}
		\sup_{F_0 \in \tilde H^\beta(\mathcal O_0): \|F_0\|_{\tilde H^\beta}\le R}\mathbb E^{\varepsilon}_{F_0}\left\|\mathscr G(\hat F)-\mathscr G(F_0)\right\|_{\mathbb H}\lesssim \delta,~~R>0,
		\end{equation}
		for any fixed sub-domain $\mathcal O_0$ such that $\bar{\mathcal O_0} \subsetneq \mathcal O$. Indeed, one easily checks (\ref{delta-cond}), and given $F_0 \in \tilde H^\beta(\mathcal O_0)$ set $F_*= \zeta \mathscr F^{-1}[ (1_{[|\cdot|\le (\delta/\lambda)^{2/d}]} \mathscr F [F_0]] \in H^\alpha_c(\mathcal O)$, where $\zeta \in C^\infty_c(\mathcal O)$ is such that $\zeta=1$ on $\mathcal O_0$ and $\mathscr F$ is the Fourier transform. Then $\|F_*\|_{H^\alpha(\mathcal O)} \lesssim \delta/\lambda$ and $\|F^*-F_0\|_{(H^\kappa(\mathcal O))^*} \lesssim \|F^*-F_0\|_{H^{-\kappa}(\mathcal O)} \lesssim \delta$ in (\ref{main-thm-rate}) yield (\ref{maprate}). Similar comments apply to non-linear $\mathscr G$, with appropriate choice of $\bar \lambda$, see Remark \ref{maphack}.
	\end{rem}
	
	\begin{ex}[Rates for the Radon transform]\label{radon} \normalfont Let $\mathscr R: \tilde {\mathcal V} \equiv L^2(\mathcal O)\to \H$ be the Radon transform, where $\mathcal O=\{x\in\mathbb R^2:\|x\|<1\}$ and $\mathbb H = L^2(\Sigma), \Sigma:=(0,2\pi] \times \mathbb R$, equipped with Lebesgue measure, see p.9~in \cite{N86} for definitions. Then $\mathscr G=\mathscr R$ satisfies (\ref{entrcond}) with $\kappa=1/2, \gamma=0$ and any $\alpha \in \mathbb N$ -- see p.42 in \cite{N86} and note that our $\|\cdot\|_{(H^{1/2}(\mathcal O))^*}$-norm is the $\|\cdot\|_{H_0^{-1/2}(\mathcal O)}$-norm used in \cite{N86} (cf.~Theorem 3.30 in \cite{ML00}). Applying Corollary \ref{cor-gen} with $\alpha\geq 1$, $\mathcal V=H^\alpha_c(\mathcal O)$ and $\lambda = \varepsilon^{(2\alpha+1)/(2\alpha+3)}$ implies that for any $F_0\in H^\alpha_c(\mathcal O)$,
		\begin{equation} \label{radfwdbd}
		\mathbb E^\eps_{F_0}\big[\|\mathscr R(\hat F_{\lambda, \eps})-\mathscr R(F_0)\|^2_{L^2(\Sigma)} + \lambda^2\|\hat F_{\lambda, \eps}\|^2_{H^\alpha_c(\mathcal O)}\big]\lesssim \varepsilon^{(4\alpha+2)/(2\alpha+3)}.
		\end{equation} 
		Using again the estimates on p.42 in \cite{N86} and that H\"older's inequality implies $$\|g\|_{H^{1/2}(\Sigma)} \le \|g\|_{L^2(\Sigma)}^{2\alpha/(2\alpha+1)} \|g\|_{H^{\alpha+1/2}(\Sigma)}^{1/(2\alpha+1)}$$ for $H^\alpha(\Sigma)$ defined as in \cite{N86}, we deduce from (\ref{radfwdbd}) and Markov's inequality that as $\varepsilon \to 0$, \begin{equation*} \|\hat F_{\lambda, \eps} - F_0\|_{L^2(\mathcal O)} \lesssim \|\mathscr R(\hat F)-\mathscr R(F_0)\|_{H^{1/2}(\Sigma)}  =O_{\mathbb P_{f_0}^{\eps}}\big(\varepsilon^{\frac{2\alpha}{2\alpha+3}} \big)\end{equation*} in probability [recall that random variables $(Z_n: n \in \mathbb N)$ are $O_{\Pr}(r_n)$ if $\forall \delta>0 ~\exists M=M_\delta$ s.t.~$\Pr(|Z_n|>Mr_n)<\delta$ for all $n \in\mathbb N$], with constants uniform in $\|F_0\|_{H_c^\alpha(\mathcal O)} \le R$ for any $R>0$. 
		Similarly, if one chooses $\lambda=\varepsilon$ instead, then the MAP estimate from Remark \ref{lin} satisfies
		\begin{equation*}
		\left\|\hat F_{\lambda, \eps}-F_0\right\|_{L^2(\mathcal O)}=O_{\mathbb P_{f_0}^{\eps}}\big(\varepsilon^{\frac{2\beta}{2\beta+3}} \big),~~\text{where }\beta:=\alpha-1>0,
		\end{equation*}
		uniformly over $\|F_0\|_{H^\beta_c(\mathcal O_0)}\leq R$ for $R> 0$.  
	\end{ex}
	
	\begin{rem}[The effect of nonlinearity] \label{nonlin} \normalfont In the proof of Theorem \ref{thm-gen} we follow ideas for $M$-estimation from \cite{saramest, sara2001}, and condition (\ref{entrcond}) is needed to bound the entropy numbers of images $\{\mathscr G(F)\;| \;\|F\|_{H^\alpha}\leq R \}, 0<R<\infty,$ of Sobolev balls under $\mathscr G$, which in turn control the modulus of continuity of the Gaussian process that determines the convergence rate of $\hat F$ to $F_0$. The at most polynomial growth in $\|F\|_{H^\alpha}$ of the Lipschitz constants
		\begin{equation}\label{gammafactor}
		\left(1+\|F\|_{H^\alpha}^\gamma\vee\|H\|_{H^\alpha}^\gamma\right), ~~ \gamma \ge 0,
		\end{equation}
		in (\ref{entrcond}) turns out to be essential in the proof of Theorem \ref{thm-gen}. But even when only a `polynomial nonlinearity' is present ($\gamma>0$), the last term in the condition (\ref{delta-cond}) can become dominant if the penalisation parameter $\lambda$ is too small. The intuition is that, for non-linear problems, too little penalisation can mean that the maximisers $\hat F$ over unbounded parameter spaces behave erratically, yielding sub-optimal convergence rates. 
	\end{rem}

	\section{Results for elliptic PDE models}\label{sec-pderes}
	In this section, we apply Theorem \ref{thm-gen} to  the inverse problems induced by the PDEs (\ref{intro-div}) and (\ref{intro-sch}). We also discuss the implied convergence rates for the parameter $f$.
	
	\subsection{Basic setup and link functions}\label{sec-pde-pre}
	For any integer $\alpha > d/2$ and any constant $K_{min}\in[0,1)$, and denoting the outward pointing normal vector at $x\in\partial \mathcal O$ by $n=n(x)$, define the parameter space (boundary derivatives are understood in the trace sense)
	\begin{equation}\label{Fdef}
	\begin{split}
	\mathcal F:=\mathcal F_{\alpha, K_{min}} =\big\{&f\in H^\alpha(\mathcal O): f>K_{min}\text{ on } \mathcal O,~f=1 \text{ on }\partial \mathcal O,\\
	&\;\;\frac{\partial^jf}{\partial n^j}=0\text{ on }\partial \mathcal O \text{ for }j=1,...,\alpha-1  \big\},
	\end{split}
	\end{equation}
	and its subclasses
	\begin{equation*}
	\mathcal F_{\alpha,r}(R):=\big\{f\in\mathcal F:f>r\text{ on } \mathcal O,\;\|f\|_{H^\alpha}\leq R\big\},~ r \ge K_{min}, R > 0.
	\end{equation*}
	We note that the restrictions $K_{min}<1$ and $f=1$ on $\partial \mathcal O$ in (\ref{Fdef}) are made only for convenience, and could be replaced by any $K_{min}>0$ and $f=\tilde g$ for \emph{fixed} $\tilde g\in C^\infty(\partial \mathcal O)$ satisfying $\tilde g>K_{min}$. Moreover, for estimation over parameter spaces without prescribed boundary values for $f$, see Remark \ref{bdhack}.
	\par 
	We will assume that the coefficient $f$ of the second order linear elliptic partial differential operators featuring in the boundary value problems (\ref{intro-div}) and (\ref{intro-sch}), respectively, belong to $\mathcal F_{\alpha, K_{min}}$ for large enough $\alpha$, and denote by  
	\begin{equation} \label{Gdef}
	G:\mathcal F\to L^2(\mathcal O), \qquad f\mapsto G(f):=u_f,
	\end{equation}
	the corresponding solution maps. Following (\ref{data}) with $\mathbb H = L^2(\mathcal O)$, we then observe
	\begin{equation}\label{data2}
	Y^{(\varepsilon)}=G(f)+\varepsilon \mathbb W, ~\varepsilon>0,
	\end{equation}
	whose law will now be denoted by $\mathbb P_{f}^\varepsilon$ for $f\in\mathcal F$. 
	
	We will apply Theorem \ref{thm-gen} to a suitable bijective re-parameterisation of $\mathcal F$ for which the set $\mathcal V$ one optimises over is a linear space. This is natural for implementation purposes but also necessary to retain the Bayesian interpretation of our estimators from Remark \ref{lin}. To this end, we introduce `link functions' $\Phi$ -- the lowercase and uppercase notation for corresponding functions $f \in\mathcal F$ and $F=\Phi^{-1}\circ f$ will be used throughout. 
	
	\begin{defin}\label{def-phi}
		1. A function $\Phi$ is called a \emph{link function} if $\Phi$ is a smooth, strictly increasing bijective map $\Phi: \mathbb R\to(K_{min},\infty)$ satisfying $\Phi(0)=1$ and $\Phi'>0$ on $\mathbb R$.
		\par
		2. A function $\Phi:(a,b)\to \mathbb R$, $-\infty\leq a<b\leq \infty$, is called \emph{regular} if all derivatives of $\Phi$ of order $k\geq 1$ are bounded, i.e.
		\begin{equation}\label{phibddder}
		\forall k\geq 1: \quad \sup_{x\in(a,b)} \left|\Phi^{(k)}(x)\right|<\infty.
		\end{equation}
	\end{defin}
	In the notation of Theorem \ref{thm-gen}, throughout this section we set $\mathbb H=L^2(\mathcal O)$, $\tilde{\mathcal V}=\mathcal V:= \{\Phi^{-1}\circ f:f\in\mathcal F \}$ to be the `pulled-back' parameter space, and
	\begin{equation}\label{G2}
	\mathscr G: \mathcal V\to  L^2(\mathcal O),\quad \mathscr G(F):=G(\Phi\circ F),
	\end{equation} 
	For $\mathcal F$ as in (\ref{Fdef}), one easily verifies that
	\begin{equation*}
	\mathcal V=\left\{F\in H^\alpha: \frac{\partial^jF}{\partial n^j}=0\text{ on }\partial \mathcal O \text{ for }j=0,...,\alpha-1  \right\}=H^\alpha_c(\mathcal O),
	\end{equation*}
	where the second equality follows from the characterization of $H^\alpha_c(\mathcal O)$ in Theorem 11.5 of \cite{lionsmagenes}. 
	%By scaling $\Phi$, we may henceforth assume $\Phi^{-1}(1)=0$, so that $\mathcal V=H^\alpha_c(\mathcal O)$. 
	Given a realisation of (\ref{data2}) and a regular link function $\Phi$, we define the \emph{generalised Tikhonov regularised functional} $J_{\lambda,\varepsilon}:\mathcal F\to \mathbb R$,
	\begin{equation}\label{Jdef}
	J_{\lambda,\varepsilon}(f):=2\langle Y^{(\varepsilon)},G(f)\rangle_{L^2}-\|G(f)\|_{L^2}^2-\lambda^2\|\Phi^{-1}\circ f\|_{H^\alpha}^2, ~~\lambda>0.
	\end{equation}
	Then for all $f\in\mathcal F$, we have $\mathscr J_{\lambda, \varepsilon}(F)=J_{\lambda,\eps}(f)$ in the notation (\ref{Jdef2}), and maximising $J_{\lambda, \eps}$ over $\mathcal F$ is equivalent to maximising $\mathscr J_{\lambda,\eps}$ over $H^\alpha_c=\mathcal V$. Any pair of maximisers will be denoted by
	\begin{equation*}
	\hat f\in \arg \max_{f\in \mathcal F}J_{\lambda, \eps}(f), ~~ \hat F=\Phi^{-1}\circ \hat f\in \arg \max_{F\in H^\alpha_c}\mathscr J_{\lambda, \eps}(F),~~G(\hat f)=\mathscr G(\hat F).
	\end{equation*}

	The proofs of the theorems which follow are based on an application of Theorem \ref{thm-gen}, after verifying that the map (\ref{G2}) satisfies (\ref{entrcond}) with $\mathcal V = H^\alpha_c$ and suitable values of $\kappa, \gamma, \alpha$. The verification of (\ref{entrcond}) is based on PDE estimates that control the modulus of continuity of the solution map  (\ref{Gdef}), and on certain analytic properties of the link function $\Phi$.  In practice often the choice $\Phi=\exp$ is made (cf.~\cite{stuart10}), but our results suggest that the use of a \emph{regular} link function might be preferable. Indeed, the polynomial growth requirement (\ref{gammafactor}) discussed above is not met if one chooses for $\Phi$ the exponential function. Before we proceed, let us give an example of a regular link function.
	
	\begin{ex} \normalfont
		Define the function $\phi: \mathbb R \to (0,\infty)$ by $\phi(x) = e^x1_{x<0} + (1+x) 1_{x \ge 0}$, let $\psi:\mathbb R\to [0,\infty)$ be a smooth, compactly supported function with $\int_\mathbb R \psi=1$, and write $\phi \ast \psi=\int_{\R}\phi(\cdot-y)\psi(y)dy$ for their convolution. It follows from elementary calculations that, for any $K_{min} \in \mathbb R$,
		\[\Phi:\R\to (K_{min},\infty),~\Phi:= K_{min}+\frac{1-K_{min}}{\psi\ast\phi(0)}\psi \ast \phi, \]
		is a regular link function with range $(K_{min},\infty)$.
	\end{ex}

	\subsection{Divergence form equation}\label{sec-div}
	For a \emph{given} source function $g \in C^\infty(\mathcal O)$, we consider the Dirichlet boundary value problem
	\begin{equation}\label{div}
	\begin{cases}
	\nabla\cdot (f\nabla u)=g\quad \textnormal{ on } \mathcal O,\\
	u=0 \quad  \textnormal{ on }\mathcal \partial O,
	\end{cases}
	\end{equation}
	where $f \in \mathcal F_{\alpha, K_{min}}$ (see (\ref{Fdef})) for some $\alpha>d/2+1, K_{min}>0$. Then (\ref{h-emb}) implies $f\in C^{1+\eta}(\mathcal O)$ for some $\eta>0,$ and the Schauder theory for elliptic PDEs (Theorem 6.14 in \cite{gt}) then gives that (\ref{div}) has a unique classical solution in $C(\bar{\mathcal O})\cap C^{2+\eta}(\mathcal O)$ which we shall denote by $G(f)=u_{f}$.

	\paragraph{Upper bounds} For a link function $\Phi$ and $f_1,f_2\in\mathcal F$, define (cf.~(\ref{taudef})) 
	\[\mu_{\lambda}(f_1,f_2):=\|G(f_1)-G(f_2)\|^2_{L^2}+\lambda^2\|\Phi^{-1}\circ f_1\|_{H^\alpha}^2=\tau_\lambda(F_1,F_2). \]
	\begin{thm}[Prediction error]\label{thm-div-pred}
		Let $\mathcal F$ be given by (\ref{Fdef}) for some integer $\alpha>(d/2+2) \vee  (2d-1)$ and $K_{min}\in (0,1)$. Let $G(f)=u_f$ denote the unique solution of (\ref{div}) and let $Y^{(\varepsilon)} \sim \P_{f_0}^\eps$ from (\ref{data2}) for some $f_0 \in \mathcal F$. Moreover, suppose that $\Phi:\mathbb R\to (K_{min}, \infty)$ is a regular link function and that $J_{\lambda_\varepsilon,\varepsilon}$ is given by (\ref{Jdef}), where
		$$\lambda_\varepsilon:=\varepsilon^{\frac{2(\alpha+1)}{2(\alpha+1)+d}}.$$
		Then the following holds.
		\begin{enumerate}
			\item For each $f_0\in\mathcal F$ and $\varepsilon>0$, almost surely under $\P_{f_0}^\eps$, there exists a maximiser $\hat f_{\varepsilon} \in \mathcal F$ of $J_{\lambda_\varepsilon,\varepsilon}$ over $\mathcal F$.
			\item For each $R>0$, $r > K_{min}$, there exist finite constants $c_1,c_2>0$ such that for any maximiser $\hat f_\varepsilon \in \mathcal F$ of $J_{\lambda_\varepsilon,\varepsilon}$, all $0<\varepsilon < 1$ and all $M\geq c_1$,
			\begin{equation}\label{div-conc}
			\sup_{f_0\in\mathcal F_{\alpha,r}(R)}\P_{f_0}^\varepsilon\Big(\mu^2_{\lambda_\varepsilon}(\hat f_\varepsilon,f_0)\geq M^2\varepsilon^{\frac{4(\alpha+1)}{2(\alpha+1)+d}}\Big)\leq \exp\Big(-\frac{M^2\lambda_\varepsilon^2}{c_2\varepsilon^2}\Big).
			\end{equation}
			\item For each $R>0$, $r > K_{min}$ and $\beta\in [0,\alpha+1]$, there exists a constant $c_3$ such that for any maximiser $\hat f_\varepsilon \in \mathcal F$ of $J_{\lambda_\varepsilon,\varepsilon}$ with corresponding $u_{\hat f_\varepsilon}$, for all $0<\varepsilon<1$,
			\begin{equation} \label{finprat}
			\sup_{f_0\in\mathcal F_{\alpha,r}(R)}\mathbb E^{\varepsilon}_{f_0}\left\|u_{\hat f_\varepsilon}-u_{f_0}\right\|_{H^\beta}\leq c_3\varepsilon^{\frac{2(\alpha+1-\beta)}{2(\alpha+1)+d}}.
			\end{equation}
		\end{enumerate}
	\end{thm}

	\paragraph{Lower bounds} We now give a minimax lower bound on the rate of estimation for $u_f$ which matches the bound in (\ref{finprat}). To facilitate the exposition we only consider the unit ball $\mathcal O=D:=\big\{x\in\mathbb R^d: \|x\|< 1 \big\}$, set $g=1$ identically on $\mathcal O$, and fix $H^\beta$-loss with $\beta=2$.
	\begin{thm}\label{thm-div-lb}
		For $K_{min}\in (0,1), \alpha>d/2+1$, $\mathcal O = D$ and $g=1$ on $\mathcal O$, consider solutions $u_f, f \in \mathcal F,$ to (\ref{div}). Then there exists $C<\infty$ such that for all $\varepsilon>0$ small enough,
		\begin{equation}\label{div-lower-bd}
		\inf_{\hat u_\varepsilon}\sup_{f_0\in \mathcal F_{\alpha, r}(R)}\mathbb E_{f_0}^\varepsilon\|\hat u_\varepsilon- u_{f_0}\|_{H^2}\geq C\varepsilon^{\frac{2(\alpha-1)}{2(\alpha+1)+d}},~r>K_{min}, R>0,
		\end{equation}
		where the infimum ranges over all measurable functions $\hat u_\varepsilon= \hat u(Y^{(\eps)})$ of $Y^{(\eps)}$ from (\ref{data2}) that take values in $H^2$.
	\end{thm}
	
	Observe that (\ref{div-lower-bd}) coincides with the lower bound for estimating $u_{f_0}$ as a regression function without PDE-constraint in $H^{\alpha+1}$ under $H^2$-loss. Note however that unconstrained `off the shelf' regression function estimators $\tilde u_\eps$ for $u_f$ will not satisfy the non-linear PDE constraint $\tilde u =G(\tilde f)$ for some $\tilde f \in \mathcal F$, thus providing no recovery of the PDE coefficient $f_0$ itself.
	\smallskip
	\paragraph{Rates for $f$ via stability estimates}
	
	For estimators $u_{\hat f_\eps}$ that lie in the range of the forward map $G$, we can resort to `stability estimates' which allow to control the convergence rate of $\hat f_\eps$ to $f_0$ by the rate of $G(\hat f_\eps)=u_{\hat f_\eps}$ towards $G(f_0)=u_{f_0}$, in appropriate norms. Injectivity and global stability estimates for this problem have been studied in several papers since Richter \cite{R81}, see the recent contribution \cite{devore} and the discussion therein. They require additional assumptions, a very common choice being that $g>0$ throughout $\bar {\mathcal O}$. The usefulness of these estimates depends in possibly subtle ways on the class of $f$'s one constrains the problem to. The original stability estimate given in \cite{R81} controls $\|f_1-f_2\|_\infty$ in terms of $\|u_{f_1}-u_{f_2}\|_{C^2}$ which does not combine well with the $H^\beta$- convergence rates obtained in Theorem \ref{thm-div-pred}. The results proved in \cite{devore} are designed for `low regularity' cases where $\alpha \in (0,1)$: they give at best	
	\begin{equation}\label{devore-stab}
	\|f_1-f_2\|_{L^2}\leq C(f_1,f_2) \|u_{f_1}-u_{f_2}\|_{H^1}^{1/2}, \qquad f_1,f_2\in\mathcal F, ~d \ge 2,
	\end{equation}
	which via Theorem \ref{thm-div-pred} would imply a convergence rate of $\eps^{\frac{\alpha}{2(\alpha+1)+d}}$ for $\|\hat f_\eps-f_0\|_{L^2}$. For higher regularity $\alpha \ge 2$ relevant here, this can be improved. We prove in Lemma \ref{lem-div-stab} below a Lipschitz stability estimate for the map $u_f\mapsto f$ between the spaces $H^2$ and $L^2$, and combined with Theorem \ref{thm-div-pred} this gives the following rate bound for $\hat f_\eps - f_0$.
	\begin{thm}\label{thm-div-f}
		Suppose that $\alpha$, $K_{min}$, $\mathcal F$, $G$, $\Phi$, $\lambda_\varepsilon$ are as in Theorem \ref{thm-div-pred} and that in addition, $\inf_{x\in\mathcal O} g(x)\ge g_{min}$ for some $g_{min}>0$. Let $\hat f_\varepsilon \in \mathcal F$ be any maximiser of $J_{\lambda_\varepsilon,\varepsilon}$. Then, for each $r >K_{min}$ and $R<\infty$, there exists a constant $C>0$ such that we have for all $0 < \eps <1$,
		\begin{equation}\label{div-f}
		\sup_{f_0\in\mathcal F_{\alpha,r}(R)}\E^{\varepsilon}_{f_0}\|\hat f_\varepsilon-f_0\|_{L^2}\leq C\varepsilon^{\frac{2(\alpha-1)}{2(\alpha+1)+d}}.
		\end{equation}
	\end{thm}
	The rate in Theorem \ref{thm-div-f} is strictly better than what can be obtained from (\ref{devore-stab}), or by estimating $\|u_{\hat f_\eps}-u_{f_0}\|_{C^2}$ by $\|u_{\hat f_\eps}-u_{f_0}\|_{H^{2+d/2+\eta}}, \eta>0,$ and using Richter's stability estimate. A more detailed study of the stability problem, and of the related question of optimal rates for estimating $f$, is beyond the scope of the present paper and will be pursued elsewhere.

	\subsection{Schr\"odinger equation}\label{sec-schr}
	We now turn to the Schr\"odinger equation 
	\begin{equation}\label{schroedeq}
	\begin{cases}
	\Delta u -2fu =0 \quad \textnormal{on }  \mathcal O,\\
	u=g\quad \textnormal{on } \mathcal \partial O,
	\end{cases}
	\end{equation} 
	with \emph{given} $g\in C^\infty(\partial \mathcal O)$. By standard results for elliptic PDEs (Theorem 6.14 in \cite{gt}), for $f \in \mathcal F_{\alpha, K_{min}}$ from (\ref{Fdef}) with $K_{min}\geq 0, \alpha>d/2$,  a unique classical solution $u_f=G(f)$ to (\ref{intro-sch}) exists which lies in $C^{2+\eta}(\mathcal O)\cap C^0(\bar{\mathcal O})$ for some $\eta>0$.
	
	The results for this PDE are similar to the previous section, although the convergence rates are quantitatively different due to the fact that the forward operator is now $2$-smoothing.
	
	\begin{thm}[Prediction error]\label{thm-schr-pred}
		Let $\mathcal F$ be given by (\ref{Fdef}) for some integer $\alpha>(d/2+2) \vee (2d-2)$
		and $K_{min}\in[0,1)$. Let $G(f)=u_f$ denote the unique solution to (\ref{schroedeq}) and let $Y^{(\varepsilon)} \sim \P_{f_0}^\eps$ from (\ref{data2}) for some $f_0 \in \mathcal F$. Moreover, suppose that $\Phi:\mathbb R\to (K_{min}, \infty)$ is a regular link function and that $J_{\lambda_\eps,\eps}$ is given by (\ref{Jdef}), where
		$$\lambda_\varepsilon=\varepsilon^{\frac{2(\alpha+2)}{2(\alpha+2)+d}}.$$
		Then the following holds.
		\begin{enumerate}
			\item For each $f_0\in\mathcal F$ and $\varepsilon>0$, almost surely under $\P_{f_0}^\eps$, there exists a maximiser $\hat f_{\varepsilon} \in \mathcal F$ of $J_{\lambda_\eps,\eps}$ over $\mathcal F$.
			\item For each $R>0$, $r > K_{min}$, there exist finite constants $c_1,c_2>0$ such that for any maximiser $\hat f_\varepsilon \in \mathcal F$ of $J_{\lambda_\eps,\eps}$, all $0<\eps<1$ and all $M\geq c_1$, we have
			\begin{equation}\label{schr-conc}
			\sup_{f_0\in\mathcal F_{\alpha,r}(R)}\P_{f_0}^\eps\Big(\mu^2_{\lambda_\eps}(\hat f_\eps,f_0)\geq M^2\eps^{\frac{4(\alpha+2)}{2(\alpha+2)+d}}\Big)\leq \exp\Big(-\frac{M^2\lambda_\eps^2}{c_2\eps^2}\Big).
			\end{equation}
			\item For each $R>0$, $r > K_{min}$ and $\beta\in [0,\alpha+2]$, there exists a constant $c_3>0$ such that for any maximiser $\hat f_\varepsilon \in \mathcal F$ of $J_{\lambda_\eps,\eps}$ and all $0<\eps<1$,
			\begin{equation}\label{schr-rate}
			\sup_{f_0\in\mathcal F_{\alpha,r}(R)}\mathbb E^{\eps}_{f_0}\left\|u_{\hat f_\varepsilon}-u_{f_0}\right\|_{H^\beta}\leq C\varepsilon^{\frac{2(\alpha+2-\beta)}{2(\alpha+2)+d}}.
			\end{equation}
		\end{enumerate}
	\end{thm}

	For the PDE (\ref{schroedeq}) the stability estimate is easier to obtain than the one required in Theorem \ref{thm-div-f}, and here is the convergence rate for estimation of $f\in\mathcal F$. We note that the rates obtained in Theorems \ref{thm-schr-pred} and \ref{thm-schr-f} are minimax-optimal in view of Proposition 2 in \cite{n17} (and its proof).
	\begin{thm}\label{thm-schr-f}
		Assume that $\alpha,K_{min}, \mathcal F,G,\Phi,\lambda_\eps$ are as in Theorem \ref{thm-schr-pred} and that in addition, $\inf_{x\in\mathcal \partial O}g(x)\geq g_{min}$ for some $g_{\min}>0$. Let $\hat f_\varepsilon \in \mathcal F$ be any maximiser of $J_{\lambda_\varepsilon,\varepsilon}$. Then for all $r > K_{min}$ and $R>0$, there exists a constant $C>0$ such that for all $\eps>0$ small enough,
		\[\sup_{f_0\in\mathcal F_{\alpha,r}(R)}\E_{f_0}^\eps\left\|\hat f_\eps-f_0 \right\|_{L^2}\leq C\eps^{\frac{2\alpha}{2(\alpha+2)+d}}. \]
	\end{thm}

	\subsection{Concluding remarks and discussion}\label{sec-pde-rem}
	
	\begin{rem}\label{dethack} \normalfont
		The classical literature on `deterministic inverse problems' deals with convergence rate questions of Tikhonov and related regularisers, see the monograph  \cite{EHN96}, \cite{EKN89, N92, SEK93, EHN96, TJ02, KNS08} and also  \cite{BB18}. The convergence analysis conducted there is typically for observations $y_\delta = \mathscr G(F)+\delta$ where $\delta$ is a fixed perturbation vector in data space, equal to $L^2(\mathcal O)$ in the setting of the present paper.  For non-linear problems, rates are obtained as $\|\delta\|\to 0$ under some invertibility assumptions on a suitable adjoint $D\mathscr G^*_{F}$ of the linearisation $D\mathscr G_{F}[\cdot]$ of the forward operator at the `true' parameter $F$ (`source conditions'), see, e.g., Section 10 in \cite{EHN96}. These results are not directly comparable since our noise $\mathbb W$ models genuine statistical error and hence is random and, in particular, almost surely \textit{not} an element of data space $L^2(\mathcal O)$. As shown in \cite{BHM04, BHMR07, HP08, LL10}, the `deterministic' analysis extends to the Gaussian regression model (\ref{data}) to a certain degree, but the results obtained there still rely, among other things, on invertibility properties of $D \mathscr G^*_{F}$. For the PDE (\ref{intro-div}) such `source conditions' are problematic as $D \mathscr G^*_{F}$ is not invertible in general (due to the fact $\nabla u_{f}$ can vanish on $\mathcal O$ unless some fairly specific further assumptions are made, see \cite{itokunisch}). Our techniques circumvent source conditions by first optimally solving the `forward problem',  and then feeding this solution into a suitable stability estimate for $\mathscr G^{-1}$.
	\end{rem}
	
	\begin{rem} [Bayesian inversion] \label{bayeshack} \normalfont
		The Bayesian approach \cite{stuart10, KvdVvZ11, DS16} to inverse problems has been very popular recently, but only few theoretical guarantees for such algorithms are available in non-linear settings: In \cite{v13}, convergence rates for the  PDE (\ref{div}) are obtained for certain Bayes procedures that arise from priors for $f$ that concentrate on specific bounded subsets of $H^\alpha$. The main idea to combine regression results with stability estimates is related to our approach, but the rates obtained in \cite{v13} are suboptimal, and for the elliptic PDE models considered here do not apply to Gaussian priors. Bayesian inference for the PDE (\ref{schroedeq}) has been studied in \cite{n17}, where it is shown that procedures based on a uniform wavelet prior \textit{do} attain minimax optimal convergence rates for $f$ and $u_f$ (up to log-factors). The paper \cite{n17} also addresses the question of uncertainty quantification via the posterior distribution, by proving nonparametric Bernstein-von Mises theorems, whereas our results only concern the convergence rate of the MAP estimate for certain Gaussian priors (see Remarks \ref{lin}, \ref{maphack}). A related recent reference is \cite{MNP19} where the asymptotics for linear functionals of Gaussian MAP estimates  are obtained in linear inverse problems involving Radon and more general $X$-ray transforms -- see also \cite{KvdVvZ11, R13} for earlier results for diagonalisable linear inverse problems. Finally, convergence rates for posterior distributions of PDE coefficients in certain non-linear parabolic (diffusion) settings have been studied in \cite{NS17, A18}.
	\end{rem}
	
	\begin{rem} [MAP estimates, non-linear $\mathscr G$] \label{maphack} \normalfont As explained in Remark \ref{nonlin}, Theorem \ref{thm-gen} does not necessarily produce optimal rates for the choice $\lambda =\varepsilon$ in the non-linear settings from this section  where $\gamma>0$. For MAP estimates as discussed in Remark \ref{lin} our results then imply optimal convergence rates for $G(f)$ only if the Gaussian prior is re-scaled in an $\varepsilon$-dependent way, more specifically if its RKHS norm is $\bar \lambda \|\cdot\|_{H^\alpha}$ with $\bar \lambda= \varepsilon^{-d/(2\alpha+2\kappa+d)}$. Moreover, positivity of $f$ is enforced by a `regular link function' $\Phi$, excluding the exponential map. Whether these restrictions on admissible priors are artefacts of our proofs remains a challenging open question, however, to the best of our knowledge, these are the first convergence rate results for proper MAP-estimators in the non-linear PDE constrained inverse problems studied here, improving in particular upon the (`sub-sequential') consistency results  \cite{dashti13}. 
	\end{rem}

	\begin{rem}\label{bdhack} \normalfont
		For both PDEs (\ref{div}) and (\ref{schroedeq}), one can also consider estimation over the parameter space
		\begin{equation}\label{ftilde}
		\tilde{\mathcal F}:= \{f\in H^\alpha(\mathcal O): \inf_{x \in \mathcal O}f(x) >K_{min}\text{ on }\mathcal O\},
		\end{equation}
		\textit{without} the boundary restrictions on $f$ from (\ref{Fdef}). Note that $\mathcal F \subset \tilde {\mathcal F}$. Then, with $\kappa=1/2-\eta$,$\eta\in (0,1/2)$, $\tilde{\mathcal V}=\mathcal V=H^\alpha(\mathcal O)$ in (\ref{G2}), $\alpha>(d/2+2) \vee 2d-\kappa$ and $K_{min}$ as before, Theorem \ref{thm-gen} applies as in Theorems \ref{thm-div-pred} and \ref{thm-schr-pred}, and  
		\[\sup_{f_0 \in \tilde{\mathcal F}_{\alpha,r}(R)} \E_{f_0}^\eps\|u_{\hat f}-u_{f_0}\|_{H^\beta(\mathcal O)}\lesssim \eps^ {\frac{2(\alpha+1/2-\eta-\beta)}{2(\alpha+1/2-\eta)+d}},\quad r>K_{min}, R>0, \]
		where, respectively, $\beta\in [0,\alpha+1]$ (divergence form eq.) and $\beta\in [0,\alpha+2]$ (Schr\"odinger eq.), and $\tilde{\mathcal F}_{\alpha,r}(R):=\{f\in\tilde{\mathcal F}: f >r \text{ on }\mathcal O,~ \|f\|_{H^\alpha} \le R\}$. By the stability Lemmas \ref{lem-div-stab} and \ref{lem-schr-stab} (which apply to $\tilde {\mathcal F}$ as well) and arguing as in the proofs of Theorems \ref{thm-div-f} and \ref{thm-schr-f}, this yields the respective convergence rates 
		\[\eps^ {\frac{2(\alpha-3/2-\eta)}{2(\alpha+1/2-\eta)+d}\cdot\frac{\alpha-1}{\alpha+1}} \;\text{ (div. form eq.)}, \quad \eps^ {\frac{2(\alpha-3/2-\eta)}{2(\alpha+1/2-\eta)+d}\cdot\frac{\alpha}{\alpha+2}} \;\text{ (Schr\"odinger eq.)},\]
		for $\E_{f_0}^\eps\|\hat f-f_0\|_{L^2}$, uniform over $\tilde{\mathcal F}_{\alpha,r}(R)$.
	\end{rem}

	\section{Proofs of the main results}\label{sec-pfs}

	\subsection{Convergence rates in $M$-estimation}

For the convenience of the reader we recall here some classical techniques for proving convergence rates for $M$-estimators (see \cite{sara2001, saramest}) -- these will form the basis for the proof of Theorem \ref{thm-gen}. In the following $\tilde{\mathcal V}\subseteq L^2(\mathcal O)$, $\mathscr G:\tilde{\mathcal V}\to \mathbb H$ is a Borel-measurable map, and the functionals $\mathscr J_{\lambda, \eps}, \tau_\lambda^2(\cdot,\cdot)$ are given by (\ref{Jdef2}), (\ref{taudef}), respectively.  Let $\mathcal V\subseteq \tilde{\mathcal V}\cap H^\alpha(\mathcal O)$ be a subset over which we aim to maximise $\mathscr J_{\lambda, \eps}$. For any $F_*\in \mathcal V$ and $\lambda, R\geq 0$, define sets
	\begin{align}\label{V*}
	\mathcal V_*(\lambda,R)&:=\{F\in\mathcal V: \tau^2_{\lambda}(F,F_*)\leq R^2  \},
	\end{align}
	their images under $\mathscr G$,
	\begin{equation}\label{D*}
	\mathcal D_*(\lambda,R)= \big\{\mathscr G(F):F\in\mathcal V \text{ with } \tau^2_{\lambda}(F,F_*)\le R^2\big\},
	\end{equation}
	and also
	\begin{equation}\label{J*}
	J_*(\lambda,R):=R+\int_0^{2R} H^{1/2}\left(\rho,\mathcal D_*(\lambda,R),\|\cdot\|_{\mathbb H}\right)d\rho,
	\end{equation}
	where the usual metric entropy of $A \subset \mathbb H$ is denoted by $H(\rho,A,\|\cdot\|_{\mathbb H} )$ ($\rho>0$). The following theorem is, up to some modifications which adapt it to the continuum sampling scheme (\ref{data}) and the inverse problem setting considered here, a version of Theorem 2.1 in \cite{sara2001}.
	\begin{thm}\label{thm-ep}
		Let $F_*\in\mathcal V, \lambda >0,$ and let $\mathbb P^\eps_{F_0}$ be the law of $Y^{(\eps)}$ from (\ref{data}) for some fixed $F_0 \in \tilde{\mathcal V}$. Suppose $\Psi_*(\lambda, R)\geq J_*(\lambda,R)$ is some upper bound such that $R\mapsto \Psi_*(\lambda, R)/R^2$ is non-increasing. Then there exist universal constants $c_1, c_2, c_3$ such that for all $\eps, \lambda, \delta>0$ satisfying
		\begin{equation}\label{delta-req}
		\delta^2\geq c_1\eps\Psi_*(\lambda, \delta)
		\end{equation}
		and any $R\geq \delta$, we have that
		\begin{align}\label{tau-conc}
		\mathbb P^\eps_{F_0}\Big(\mathscr J_{\lambda,\eps} &\text{ has a maximizer } \hat F \text{ over }\mathcal V  \text{ s.t. }\tau_\lambda^2(\hat F,F_0)\geq 2(\tau_\lambda^2(F_*,F_0)+R^2)\Big) \notag \\
		&\leq c_2\exp\Big(-\frac{R^2}{c_1^2\eps^2}\Big).
		\end{align}
		Moreover, for any maximiser $\hat F$ of $\mathscr J_{\lambda,\eps}$ over $\mathcal V$ we have for some universal constant $c_3$
		\begin{equation}\label{tau-exp}
		\E_{F_0}^\eps[\tau_\lambda^2(\hat F,F_0)]\leq c_3(\tau_\lambda^2(F_*,F_0)+\delta^2+\eps^2).
		\end{equation}
	\end{thm}
	
	\begin{proof}
		1. Let $\hat F$ denote any maximiser of $\mathscr J_{\lambda, \eps}$. By completing the square, we see that $\hat F$ also maximises
		\[Q_{\lambda,\eps}(F):=2\langle\varepsilon \mathbb W,\mathscr G( F)\rangle_{\mathbb H}-\|\mathscr G(F)-\mathscr G(F_0)\|_{\mathbb H}^2-\lambda^2\|F\|_{H^\alpha}^2.\]
		Rewriting the inequality $Q_{\lambda,\eps}(\hat F)\geq Q_{\lambda,\eps}(F_*)$, we obtain
		\[2\langle\varepsilon \mathbb W,\mathscr G( \hat F)-\mathscr G(F_*)\rangle_{\mathbb H}\geq \tau_\lambda^2(\hat F,F_0)- \tau_\lambda^2(F_*,F_0).\]
		Elementary calculations as in \cite{sara2001}, p.3-4, give that for all $R>0$, if
		\[\tau_\lambda^2(\hat F,F_0)\geq 2\left(\tau_\lambda^2(F_*,F_0)+R^2\right) \] holds 
		then we also have the inequalities
		\begin{equation*}
		\begin{split}
		\tau_\lambda^2(\hat F,F_*)&\geq R^2 \qquad \text{and}\\
		\tau_\lambda^2(\hat F,F_0)-\tau_\lambda^2(F_*,F_0)&\geq \frac 16 \tau_\lambda^2(\hat F,F_*).
		\end{split}
		\end{equation*}
		It follows that for any $R>0$ and for $\mathbb P$ the law of the centred Gaussian process $(\mathbb W(\psi)=\langle \mathbb W, \psi \rangle_{\mathbb H}: \psi \in \mathbb H)$,
		\begin{equation*}
		\begin{split}
		\P_{F_0}^\eps&\left(\tau_\lambda^2(\hat F,F_*)\geq 2\left(\tau_\lambda^2(F_*,F_0)+R^2\right)\right)\\
		&\le \P_{F_0}^\eps \Big(\tau_\lambda^2(\hat F,F_*)\geq R^2,~2\langle\varepsilon \mathbb W,\mathscr G( \hat F)-\mathscr G(F_*)\rangle_{\mathbb H}\geq \frac 16 \tau_\lambda^2(\hat F,F_*) \Big)\\
		&\leq \sum_{l=1}^{\infty}\P\left(\sup_{\psi\in \mathcal D_*(\lambda,2^lR)}\langle\varepsilon \mathbb W,\psi-\mathscr G(F_*)\rangle_{\mathbb H} \geq \frac{1}{48}2^{2l}R^2\right) =:\sum_{l=1}^\infty P_l.
		\end{split}
		\end{equation*}
		\par
		2. For all $\lambda, R\geq 0$, we have that $\sup_{\psi,\varphi\in\mathcal D_*(\lambda,R)}\|\psi-\varphi\|_{\mathbb H}\leq 2R,$
		so that by Dudley's theorem (see \cite{nicklgine}, p.43),
		\begin{equation*}
		\begin{split}
		&\E\left[\sup_{\psi\in\mathcal D_*(\lambda,R)}|\langle \W,\psi-\mathscr G(F_*)\rangle_
		\mathbb H|\right]\\
		&\qquad \lesssim \inf_{\psi\in \mathcal D_*(\lambda,R)}\E|\langle \W,\psi-\mathscr G(F_*)\rangle_\mathbb H|+\int_0^{2R}H^{1/2}\left(\rho,\mathcal D_*(\lambda,R),\|\cdot\|_{\mathbb H}\right)d\rho\\
		&\qquad \lesssim R+\int_0^{2R}H^{1/2}\left(\rho,\mathcal D_*(\lambda,R),\|\cdot\|_{\mathbb H}\right)d\rho= J_*(\lambda, R)\leq \Psi_*(\lambda,R).
		\end{split}
		\end{equation*}

		\par
		3. Let us write $S_*(\lambda,R):=\sup_{\psi\in\mathcal D_*(\lambda,R)}|\langle \W,\psi-\mathscr G(F_*)\rangle|$. By choosing $c$ large enough and $\delta$ such that (\ref{delta-req}) holds, we have that for all $R\geq \delta$, $\frac{1}{48}R^2-\eps\Psi_*(\lambda,R)\geq \frac{1}{96}R^2$. Thus by the preceding display, the Borell-Sudakov-Tsirelson inequality (see Theorem 2.5.8 in \cite{nicklgine}), and possibly making $c>0$ larger, we obtain for all $R\geq \delta$ and $l=1,2,...$ 
		\begin{equation}\label{pl-est}
		\begin{split}
		P_l&\leq \P\left( \eps S_*(\lambda,2^lR)-\eps \E[S_*(\lambda,2^lR)] \geq \frac{1}{48}2^{2l}R^2-\eps\Psi_*(\lambda, 2^lR)\right)\\
		&\leq \P\left( S_*(\lambda,2^lR)-\E[S_*(\lambda,2^lR)] \geq \frac{2^{2l}R^2}{96\eps}\right)\\
		&\leq \exp\left(-\frac 12\left(\frac{2^{2l}R^2}{96\eps}\right)^22^{-2l}R^{-2}\right) \leq \exp\left(-\frac{2^{2l}R^2}{c\eps^2}\right),
		\end{split}
		\end{equation}
		where in the penultimate inequality, we have used 
		$$\sup_{\psi\in\mathcal D_*(\lambda, 2^lR)}\E[|\langle\W,\psi-\mathscr G(F_*)\rangle_\mathbb H|^2]\leq 2^{2l}R^2.$$
		The inequality (\ref{tau-conc}) now follows from summing (\ref{pl-est}), and (\ref{tau-exp}) follows from arguing as in the proof of Lemma 2.2 in \cite{sara2001}.
	\end{proof}
	
	\subsection{Proof of \ref{thm-gen}, Part 2}\label{sec-pfs-sec2}
	We will apply Theorem \ref{thm-ep} and need the following lemma. For  $F_* \in \mathcal V$, define $\mathcal V_*(\lambda, R),\mathcal D_*(\lambda, R)$ and $J_*(\lambda,R)$ by (\ref{V*}), (\ref{D*}) and (\ref{J*}) respectively. We also use the notation $H^\alpha(\mathcal O,r):=\{F\in H^\alpha(\mathcal O)\;|\; \|F\|_{H^\alpha}\leq r\}$ and $H^\alpha_c(\mathcal O,r):=\{F\in H^\alpha_c(\mathcal O)\;|\; \|F\|_{H^\alpha}\leq r\}, r >0$ and recall $s=(\alpha+\kappa)/d$.

	\begin{lem}\label{lem-main}
		Suppose that $\mathcal V$ and $\mathscr G$ are as in Part 2 of Theorem \ref{thm-gen}. Then there exists a positive constant $c$ such that for all $\lambda, R>0$ and $F_*\in \mathcal V$,
		%$F_*\in \arg\min_{F\in\mathcal V}\tau_\lambda(F,F_0)$ (or if no minimizer exists, for all $F_*$ such that $\tau_\lambda(F_*,F_0)$ is sufficiently close to the infimum),
		\begin{equation*}
		\begin{split}
		\Psi_*(\lambda,R)&:=R+c\big(R \lambda^{-\frac{1}{2s}}\big(1+(R/\lambda)^{\gamma/2s}\big)\big) 
		\end{split}
		\end{equation*}
		is an upper bound for $J_*(\lambda,R)$.
		% such that $R\mapsto\Psi_*(\lambda,R)/R^2$ is non-increasing.
	\end{lem}
	\begin{proof}
		Let us first assume that $\kappa \geq 1/2$. We estimate the metric entropy in $J_*(\lambda,R)$. Let $\rho, \lambda, R>0$ and define 
		\begin{equation*}\label{mk}
		m:=C\left(1+R^\gamma\lambda^{-\gamma}\right),
		\end{equation*}
		where $C$ is the constant from (\ref{entrcond}). By definition of $\tau_\lambda$, we have $\mathcal V_*(\lambda,R)\subseteq H^\alpha_c(\mathcal O,R/\lambda)$. Fix some larger, bounded $C^\infty$-domain $\tilde{\mathcal O}\supset \bar{\mathcal O}$ and some function $\zeta\in C_c^\infty(\R^d)$ such that $0\leq \zeta\leq 1$, $\zeta=1$ on $\mathcal O$ and $\text{supp}(\zeta)\subset \tilde{\mathcal O}$. By the main theorem of Section 4.2.2 in \cite{T78}, there exists a bounded, linear extension operator $\mathcal E:H^{\kappa}(\mathcal O)\to H^\kappa(\R^d)$. Define the map $e:\phi\mapsto \zeta \mathcal E(\phi)$ which maps $H^\kappa(\mathcal O)$ continuously into $\tilde H^\kappa(\tilde {\mathcal O})$, and for $\phi\in L^2(\mathcal O)$, let $\tilde \phi:\R^d\to \R$ denote its extension by $0$ on $\R^d\setminus \mathcal O$. We then have, for some $c_1>0$,
		\begin{equation}\label{neg-sob}
		\|\phi\|_{(H^\kappa(\mathcal O))^*}=\sup_{\varphi\in H^\kappa(\mathcal O,1)}\left|\int_{\mathcal O}\phi\varphi \right|= \sup_{\varphi\in  H^\kappa(\mathcal O,1)}\left|\int_{\tilde{\mathcal O}}\tilde{\phi}e(\varphi) \right|\le c_1\|\tilde{\phi}\|_{H^{-\kappa}(\tilde{\mathcal O})}.
		\end{equation}
		By Theorem 11.4 in \cite{lionsmagenes} and its proof, the zero extension $\phi\mapsto \tilde \phi$ is continuous from $H^\alpha_c(\mathcal O)$ to $H^\alpha(\tilde{\mathcal O})$ with norm $1$, so that 
		\begin{equation*}
		\mathcal W:=\left\{\tilde F: F\in \mathcal V_*(R,\lambda) \right\}\subseteq H^\alpha_c(\tilde{\mathcal O},R/\lambda).
		\end{equation*}
		By Theorem 4.10.3 of \cite{T78}, we can pick $\tilde F_1,...,\tilde F_N\in \mathcal W$ with
		\[N\leq \exp\Big(c_2\big(\frac{Rmc_1}{\lambda\rho}\big)^{\frac 1s}\Big)\]
		for some universal constant $c_2$, such that the balls 
		\[\tilde B_i:=\Big\{\psi\in \mathcal W\;: \; \|\psi-\tilde F_i\|_{H^{-\kappa}(\tilde {\mathcal O})} \leq \frac{\rho}{mc_1}\Big\}, \qquad i=1,...,N,\]
		form a covering of $\mathcal W$. Then it follows from (\ref{entrcond})  and $(\ref{neg-sob})$ that for all $i=1,...,N$ and $F$ with $\tilde F\in \tilde B_i$,
		\[\|\mathscr G(F)-\mathscr G(F_i)\|_{\H}\leq m\|F-F_i\|_{(H^{\kappa}(\mathcal O))^*}\leq mc_1\|\tilde F-\tilde F_i\|_{H^{-\kappa}(\tilde{\mathcal O})}, \]
		whence the balls 
		%\left\{\mathscr G(F):F\in\mathcal V_*(R,\lambda),\tilde F\in \tilde{B}_i \right\}\subseteq  
		\[B_i' :=\{\psi\in \mathcal D_*(\lambda,R): \|\psi-\mathscr G(F_i)\|_{\mathbb H}\leq \rho \},\quad i=1,...,N\]
		form a covering of $\mathcal D_*(\lambda,R)$. Hence we obtain the bound
		\begin{equation}\label{entrbound}
		H\left(\rho, \mathcal D_*(\lambda, R),\|\cdot\|_{\mathbb H}\right)\lesssim \Big(\frac{Rm}{\lambda\rho}\Big)^{\frac 1s}, 
		\end{equation}
		and hence also
		\begin{equation*}
		\begin{split}
		\int_0^{2R} H^{1/2}\left(\rho, \mathcal D_*(\lambda, R),\|\cdot\|_{\mathbb H}\right)d\rho &\lesssim \int_0^{2R}\left(\frac{Rm}{\lambda\rho}\right)^{\frac 1{2s}}d\rho \lesssim R\lambda^{-\frac{1}{2s}}(1+(R/\lambda)^{\frac{\gamma}{2s}}),
		\end{split}
		\end{equation*}
		which proves that $\Psi_*\geq J_*$ for the case $\kappa\geq 1/2$.
		\par
		For $\kappa<1/2$, by Theorem 11.1 in \cite{lionsmagenes}, we have $\tilde H^\kappa(\mathcal O)=H^\kappa_c(\mathcal O)=H^\kappa(\mathcal O)$ and hence $\|\cdot\|_{(H^\kappa(\mathcal O))^*}=\|\cdot\|_{H^{-\kappa}(\mathcal O)}$, whence we can use Theorem 4.10.3 of \cite{T78} directly to cover $\mathcal V_*(R,\lambda)\subseteq H^\alpha(\mathcal O,R/\lambda)$ by $H^{-\kappa}(\mathcal O)$-balls, and using (\ref{entrcond})  as above yields the entropy bound (\ref{entrbound}).
	\end{proof}
	
	By assumption on $\alpha$ we have $1+\frac{\gamma}{2s}< 2$ and hence the map $R\mapsto\Psi_*(\lambda,R)/R^2$, for $\Psi_*(\lambda,R)$ as defined in Lemma \ref{lem-main}, is decreasing. The bounds (\ref{main-thm-tau}) and (\ref{main-thm-rate}) then follow from Theorem \ref{thm-ep}. The proof of existence of maximisers is given in Section \ref{sec-ex-pf}.  Finally, we obtain Corollary \ref{cor-gen} by taking $F_*=F_0$ and $\delta:=c\eps^{2(\alpha+\kappa)/(2\alpha+2\kappa+d)}$ for which (\ref{delta-cond}) is easily verified for $\lambda$ chosen as in the corollary, so that Theorem \ref{thm-gen} applies.

	\subsection{Proof of Theorems \ref{thm-div-pred}, \ref{thm-div-lb} and \ref{thm-div-f}}\label{sec-div-pfs}

	\begin{proof}[Proof of Theorem \ref{thm-div-pred}] 
		We  verify that $\mathscr G$ given by (\ref{G2}) with $G$ the solution map of (\ref{div}), satisfies (\ref{entrcond}) for $\mathcal V = H^\alpha_c, \mathbb H = L^2(\mathcal O)$, $\gamma =4, \kappa=1$, in order to apply Theorem \ref{thm-gen}. Let $F,H\in H^\alpha$, and let us write $f:=\Phi\circ F$, $h:=\Phi\circ H$. With $L_f, V_f$ introduced in Section \ref{sec-div-facts} we have by (\ref{G2}) and (\ref{div})
		\begin{equation} \label{obpertu}
		\begin{split}
		&L_f[\mathscr G(F)-\mathscr G(H)]=L_f[u_f-u_h]\\
		&\quad =L_f[u_f]-L_h[u_h]+(L_h-L_f)[u_h]=\nabla\cdot ((h-f)\nabla u_h),
		\end{split}
		\end{equation}
		and then, by Lemma \ref{lem-div-h2} with $H^2_0$ defined in (\ref{h0}), the estimate
		\begin{equation}\label{diventr1}
		\begin{split}
		\|\mathscr G(F)-\mathscr G(H)\|_{L^2}&=\left\|V_f\left[\nabla\cdot ((h-f)\nabla u_h)\right]\right\|_{L^2}\\
		&\leq C(1+\|f\|_{C^1})\left\|\nabla\cdot ((h-f)\nabla u_h)\right\|_{(H^2_0)^*}.
		\end{split}
		\end{equation}
		By applying the divergence theorem to the vector field $\varphi (h-f)\nabla u_h$, where $\varphi\in C^2_0$ is any $C^2$-function that vanishes at the boundary, we have
		\begin{equation*}
		\begin{split}
		\left\|\nabla\cdot ((h-f)\nabla u_h)\right\|_{(H^2_0)^*}&=\sup_{\varphi\in C_0^2,\; \|\varphi\|_{H^2\leq 1}}\left|\int_{\mathcal O}\varphi \nabla\cdot ((h-f)\nabla u_h) \right|\\
		&=\sup_{\varphi\in C_0^2,\; \|\varphi\|_{H^2\leq 1}}\left|\int_{\mathcal O}(h-f)\nabla \varphi \cdot \nabla u_h \right|\\
		&\leq \|h-f\|_{(H^1)^*}\sup_{\varphi\in C_0^2,\; \|\varphi\|_{H^2\leq 1}}\|\nabla \varphi \cdot \nabla u_h\|_{H^1}\\
		&\lesssim \|h-f\|_{(H^1)^*}\|u_h\|_{\mathcal C^2},
		\end{split}
		\end{equation*}
		where we used the multiplicative inequality (\ref{c-h-mult}) in the last step. Combining this with (\ref{diventr1}) and Lemma \ref{lem-div-c} yields that
		\begin{equation*}
		\begin{split}
		\|\mathscr G(F)-\mathscr G(H)\|_{L^2}&\lesssim (1+\|f\|_{C^1})(1+\|h\|_{C^1}^2)\|h-f\|_{(H^1)^*}. 
		\end{split}
		\end{equation*}
		Hence, by (\ref{phi-cm}), (\ref{phi-neg-sob}) and the Sobolev embedding (\ref{h-emb}), we obtain
		\begin{equation*}
		\begin{split}
		\|\mathscr G(F)-\mathscr G(H)\|_{L^2}\lesssim (1+\|F\|_{H^\alpha}^4\vee\|H\|_{H^\alpha}^4)\|F-H\|_{(H^1)^*},
		\end{split}
		\end{equation*}
		so $\mathscr G$ indeed fulfills (\ref{entrcond}) for $\gamma=4$ and $\kappa=1$.
		\par
		The existence of maximisers $\hat f_\eps$ now follows from the first part of Theorem \ref{thm-gen}, and we prove (\ref{div-conc}) by applying Theorem \ref{thm-gen} with $F_*=F_0$. First, we note that for all $\hat f_\eps$ and $f_0$,
		\begin{equation}\label{mu-tau}
		\mu_{\lambda}(\hat f_\eps,f_0)= \tau_{\lambda}(\hat F_\eps, F_0).
		\end{equation}
		For the choice $\delta_\eps=c\eps^{\frac{2(\alpha+1)}{2(\alpha+1)+d}}$ and $c$ large enough, the triple $(\eps,\lambda_\eps,\delta_\eps)$ satisfies (\ref{delta-cond}) and Theorem \ref{thm-gen} and (\ref{mu-tau}) yield that for some $c'>0$ and any $m\geq \delta_\eps$,
		\[\P_{f_0}^\eps\left(\mu^2_{\lambda_{\eps}}(\hat f_\eps,f_0)\geq 2(\delta_\eps^2+m^2) \right)\leq \exp\left(-\frac{m^2}{c'\eps^2}\right), \]
		which proves (\ref{div-conc}).
		\par
		To show (\ref{finprat}), let now $\beta\in[0,\alpha+1]$, $R>0$ and $r > K_{min}$. By Lemma \ref{lem-div-bd}, we have that 
		\[M:=\sup_{f\in\mathcal F:\|f\|_{H^\alpha}\leq R}\|u_f\|_{H^{\alpha+1}}<\infty. \]
		Now for any $f_0\in\mathcal F_{\alpha,r}(R)$, we can use (\ref{h-int}) to estimate
		\begin{equation}\label{est-1}
		\begin{split}
		\|u_{\hat f}-u_{f_0}\|_{H^\beta}&\lesssim\|u_{\hat f}-u_{f_0}\|_{L^2}^{\frac{\alpha+1-\beta}{\alpha+1}}\|u_{\hat f}-u_{f_0}\|_{H^{\alpha+1}}^{\frac{\beta}{\alpha+1}}\\
		&\lesssim \|u_{\hat f}-u_{f_0}\|_{L^2}^{\frac{\alpha+1-\beta}{\alpha+1}}\left(M^{\frac{\beta}{\alpha+1}}+\|u_{\hat f}\|_{H^{\alpha+1}}^{\frac{\beta}{\alpha+1}}\right).
		\end{split}
		\end{equation}
		Further, Lemma \ref{lem-div-bd} and (\ref{phi-sob}) yield that 
		\begin{equation}\label{est-2}
		\|u_{\hat f}\|_{H^{\alpha+1}}^{\frac{\beta}{\alpha+1}}\lesssim 1+\|\hat f\|_{H^\alpha}^{\alpha\beta}\lesssim 1+\|\hat F\|_{H^\alpha}^{\alpha^2\beta}\lesssim 1+\left(\lambda_{\eps}^{-1}\mu_{\lambda_\eps}(\hat f,f_0)\right)^{\alpha^2\beta}. 
		\end{equation}
		Now set $\delta_\eps:=c_1\eps^{\frac{2(\alpha+1)}{2(\alpha+1)+d}}$ for $c_1$ from the second part of the theorem. We define the events 
		\begin{equation}\label{Aj}
		\begin{cases}
		A_0:=\{\mu_{\lambda_\eps}(\hat f_\eps,f_0)< \delta_\eps\}\\
		A_j:=\{\mu_{\lambda_\eps}(\hat f_\eps,f_0)\in (2^{j-1}\delta_\eps,2^j\delta_\eps]\}, \quad j\geq 1 .
		\end{cases}
		\end{equation}
		By (\ref{div-conc}) and (\ref{est-1})-(\ref{est-2}), and writing $\hat \mu_{\lambda_\eps}:=\mu_{\lambda_\eps}(\hat f_\eps,f_0)$, we then obtain
		\begin{equation}\label{est-3}
		\begin{split}
		&\E_{F_0}^\eps\left[\|u_{\hat f}-u_{f_0}\|_{H^\beta} \right]\lesssim \sum_{j=0}^\infty\E_{F_0}^\eps\left[ 1_{A_j}\|u_{\hat f}-u_{f_0}\|_{L^2}^{\frac{\alpha+1-\beta}{\alpha+1}}\left(1+\lambda_{\eps}^{-\alpha^2\beta}\hat \mu_{\lambda_\eps}^{\alpha^2\beta}\right) \right]\\
		&\lesssim \delta_\eps^{\frac{\alpha+1-\beta}{\alpha+1}}+\sum_{j=1}^\infty  (2^{j}\delta_\eps)^{\frac{\alpha+1-\beta}{\alpha+1}}\left(1+\lambda_{\eps}^{-\alpha^2\beta}(2^{j}\delta_\eps)^{\alpha^2\beta}\right)\P_{f_0}^\eps\left(A_j\right)\\
		&\lesssim \delta_\eps^{\frac{\alpha+1-\beta}{\alpha+1}}\Big(1+\sum_{j=1}^\infty 2^{\frac{j(\alpha+1-\beta)}{\alpha+1}}\big(1+(c2^j)^{\alpha^2\beta}\big)\exp \big(-\frac{2^{2j}\delta_\eps^2}{c_2^2\eps^2}\big) \Big) \\
		&\lesssim \delta_\eps^{\frac{\alpha+1-\beta}{\alpha+1}}(1+o(\eps)),
		\end{split}
		\end{equation}
		where $c_2$ is the constant from (\ref{div-conc}). The theorem is proved.
	\end{proof}

	\begin{proof}[Proof of Theorem \ref{thm-div-f}]
		We apply Lemma \ref{lem-div-stab} with $f_2=\hat f $ and $f_1 = f_0 \in \mathcal F_{\alpha,r}(R)$, so that $\|u_{f_1}\|_{C^1} \vee \|f_1\|_{C^1}$ is bounded by some fixed $B=B(R)$ (cf.~(\ref{h-emb}) and Lemma \ref{lem-div-c}). Thus, writing $\hat F_\eps:=\Phi^{-1}\circ \hat f_\eps$ and using (\ref{phi-cm}),
		\begin{equation*}
		\begin{split}
		\E_{f_0}^\eps\|\hat f_\eps-f_0\|_{L^2}&\lesssim \E_{f_0}^\eps\left[\|u_{\hat f_\eps}-u_{f_0}\|_{H^2} \|\hat f_\eps\|_{C_1}\right]\\
		&\lesssim \E_{f_0}^\eps\left[\|u_{\hat f_\eps}-u_{f_0}\|_{L^2}^{\frac{(\alpha-1)}{\alpha+1}}\|u_{\hat f_\eps}-u_{f_0}\|_{H^{\alpha+1}}^{\frac{2}{\alpha+1}} (1+\|\hat F_\eps\|_{C_1})\right].
		\end{split}
		\end{equation*}
		We now choose $\delta_\eps:=c_1\eps^{\frac{2(\alpha+1)}{2(\alpha+1)+d}}$ where $c_1$ is the constant from the second part of Theorem \ref{thm-div-pred}.  Bounding $\|u_{\hat f_\eps}-u_{f_0}\|_{H^{\alpha+1}}$ as in (\ref{est-1})-(\ref{est-2}), splitting the expectation into $A_j$, $j\geq 0$ as defined in (\ref{Aj}) and using the concentration inequality (\ref{div-conc}), we obtain as in (\ref{est-3}) the desired inequality
		\begin{equation*}
		\begin{split}
		\E_{f_0}^\eps\|\hat f_\eps-f_0\|_{L^2}\lesssim \delta_\varepsilon^{\frac{\alpha-1}{\alpha+1}}(1+o(\eps)).
		\end{split}
		\end{equation*}
		
	\end{proof}

	\begin{proof}[Proof of Theorem \ref{thm-div-lb}] We only prove the more difficult case $d \ge 2$.
		\par
		1. Let $f_0=1$. By direct computation, one verifies that the unique classical solution to (\ref{div}) with $g=1, \mathcal O=D$ is 
		$$u_{f_0}(x)=\frac{1}{2d}\left(\|x\|^2-1\right),\qquad \nabla u_{f_0}(x)=\frac{x}{d}.$$
		Thus we have that for some $1/2<a<b<1$,
		\begin{equation*}
		[a,b]^d\subset D,\qquad \frac 1 {2d}\leq \partial_{x_i}u_{f_0}(x)\leq \frac 1d\quad \textnormal{for all} \;\; i=1,...,d\;\;\textnormal{and} \;\; x\in [a,b]^d.
		\end{equation*}
		\par
		2. Now let $\Psi:\mathbb R\to\mathbb R$ be a $1$-dimensional, compactly supported, at least $(\alpha+1)$-regular Daubechies wavelet (see \cite{nicklgine}, Theorem 4.2.10). Then, for all integers $j\geq 1$, for suitable constants $n_j$, $c>1$ and shift vectors $v^{j,r}=(v^{j,r}_1,...,v^{j,r}_d)$ to be chosen later, we define the tensor wavelets $\Psi_{j,r}$, $r=1,...,n_j$ by
		\[\Psi_{j,r}(x)=2^{\frac{jd}{2}}c^{-\frac{d-1}{2}}\Psi(2^jx_1+v^{j,r}_1)\prod_{i=2}^d\Psi\left(\frac{2^j}cx_i+v^{j,r}_i\right).\]
		Note that the $\Psi_{j,r}$ are `steeper' by a fixed constant $c$ in $x_1$-direction than in any other direction. Due to the compact support of $\Psi$, there exists a constant $c_0$ which depends only on $c$ and $\Psi$ such that for all $j\geq j_0$ large enough, we can set $n_j=c_02^{jd}$ and find suitable vectors $v^{j,r}$ such that all $\Psi_{j,r}$ are supported in the interior $[a,b]^d$ with disjoint support. For some sufficiently small constant $\kappa>0$, we define
		\begin{equation}\label{fm}
		f_m:=f_0+\kappa 2^{-j(\alpha+d/2)}\sum_{r=1}^{n_j}\beta_{r,m}\Psi_{j,r},\qquad m=1,...,M,
		\end{equation}
		where $\beta_{r,m},\;m=1,...,M$ will be chosen later as a suitably separated elements of the hypercube $\beta_r\in\{-1,1\}^{n_j}$. 
		\par
		3. We choose $\kappa$ small enough (independently of $c>1$), as follows. By the wavelet characterisation of Sobolev norms, all $f_m$ of the form (\ref{fm}) lie in a fixed $H^\alpha$-ball of radius $C\kappa$, for some universal constant $C>0$, in particular $\|f_m-f_0\|_\infty$ can be made as small as desired for $\kappa$ small enough, so that all the $f_m>K_{min}$. Arguing as in (\ref{obpertu}), using $L_{f_0}=\Delta$ (the standard Laplacian), (\ref{c2-isom}), the multiplicative inequality (\ref{c-mult}), Lemma \ref{lem-div-c} and the Sobolev embedding $H^\alpha\subseteq C^{1+\eta}$ (for some small $\eta>0$), we have (uniformly for all $f_m$)
		\begin{equation*}
		\begin{split}
		\|u_{f_m}-u_{f_0}\|_{\mathcal C^2}&=\|V_{f_0}\left[\nabla\cdot \left((f_m-f_0)\nabla u_{f_m}\right)\right]\|_{\mathcal C^2}\\
		&\lesssim \|\nabla\cdot \left((f_m-f_0)\nabla u_{f_m}\right)\|_{\mathcal C^0}\\
		&\lesssim \|(f_m-f_0)\nabla u_{f_m}\|_{\mathcal C^1}\\
		&\lesssim  \|f_m-f_0\|_{\mathcal C^1}\|u_{f_m}\|_{\mathcal C^{2}}\\
		&\lesssim \|f_m-f_0\|_{H^\alpha}(1+\|f_m\|_{\mathcal C^1}^2).
		\end{split}
		\end{equation*}
		Therefore, $\sup_m \|u_{f_m}\|_{\mathcal C^2}<\infty$ and we can pick $\kappa$ so small that for all $f_m$ of the form (\ref{fm}),
		\begin{equation}\label{derivest}
		\frac 1 {4d}\leq \partial_{x_i}u_{f_m}(x)\leq \frac 2d\quad \textnormal{for all} \;\; i=1,...,d\;\;\textnormal{and} \;\; x\in [a,b]^d.
		\end{equation}
		\par 
		4. Next, we want to apply Theorem 6.3.2 from \cite{nicklgine}, for which two steps are needed: an appropriate lower bound on the $H^2$-distance between the $u_{f_m}$'s and a suitable upper bound on the KL-divergence of the laws $\mathbb P^\eps_{f_m}, \mathbb P^\eps_{f_0}$.
		\par
		5. We begin with the lower bound. By the isomorphism (\ref{lapl-isom}), for all $u\in H^2_0$ and $f\in \mathcal F$, we have that 
		\[\|u\|_{H^2}\gtrsim \|\Delta u\|_{L^2}=\|f^{-1}(L_fu-\nabla u\cdot \nabla f)\|_{L^2}\ge \|f\|^{-1}_{\infty}\|L_fu-\nabla u\cdot \nabla f\|_{L^2}. \]
		For all $m,m'=1,...,M$, using this inequality with $f=f_m$, 
		\begin{equation}\label{fmfm'}
		u=u_{f_m}-u_{f_{m'}}= V_{f_m}[\nabla\cdot (f_{m'} -f_m)\nabla u_{f_{m'}})]
		\end{equation}
		in view of (\ref{obpertu}), and $\sup_m\|f_m\|_{C^1}<\infty$,
		\begin{equation}\label{lowerbound}
		\begin{split}
		\|u_{f_m}-u_{f_{m'}}\|_{H^2} & \gtrsim \left\|\nabla\cdot \left(\left(f_m-f_{m'}\right)\nabla u_{f_{m'}}\right)\right\|_{L^2} - \|\nabla (u_{f_m}-u_{f_{m'}}) \cdot \nabla f_m\|_{L^2} \\
		&\geq \|\nabla(f_m-f_{m'}) \cdot \nabla u_{f_{m'}}\|_{L^2}-\|(f_m-f_{m'})\Delta u_{f_{m'}}\|_{L^2}  \\
		& \quad\quad -\|u_{f_m}-u_{f_{m'}}\|_{H^1} \|f_m\|_{C^1} =: I-II - III.
		\end{split}
		\end{equation}
		We will later show that the second and third terms are of smaller order than the first term. Using (\ref{derivest}), we see
		\begin{equation}\label{I-est}
		\begin{split}
		I&=\left\|\sum_{i=1}^{d}\partial_{x_i}(f_m- f_{m'})\partial_{x_i}u_{f_{m'}}\right\|_{L^2}\\
		&\geq \|\partial_{x_1}(f_m- f_{m'})\partial_{x_1}u_{f_{m'}}\|_{L^2}-\sum_{i=2}^{d}\|\partial_{x_i}(f_m- f_{m'})\partial_{x_i}u_{f_{m'}}\|_{L^2}\\
		&\geq \frac 1{4d}\|\partial_{x_1}(f_m- f_{m'})\|_{L^2}-\frac 2d\sum_{i=2}^{d}\|\partial_{x_i}(f_m- f_{m'})\|_{L^2}.
		\end{split}
		\end{equation}
		To estimate this further, we calculate that for any $i=2,...,d$,
		\begin{equation*}
		\begin{split}
		\partial_{x_i}\Psi_{j,r}(x)&=2^{\frac{jd}{2}}c^{-\frac{d-1}{2}}\Psi(2^jx_1+v^{j,r}_1)\\
		&\qquad\qquad \times \Big(\prod_{k=2,\;k\neq i}^d\Psi\Big(\frac{2^j}cx_k+v^{j,r}_k\Big)\Big)\frac{2^j}{c}\Psi'\left(\frac{2^j}cx_i+v^{j,r}_i\right). 
		\end{split}
		\end{equation*}
		Similarly calculating $\partial_{x_1}\Psi_{j,r}$ and summing over $r=1,...,n_j$, we obtain
		\[\|\partial_{x_i}(f_m-f_{m'})\|_{L^2}=\frac{1}{c}\|\partial_{x_1}(f_m-f_{m'})\|_{L^2}, \quad i=2,...,d. \]
		Thus, choosing $c$ large enough and combining this with (\ref{I-est}), we can ensure that
		\begin{equation*}
		\begin{split}
		I &\gtrsim \frac 1{4d}\|\partial_{x_1}(f_m- f_{m'})\|_{L^2}-\frac{2(d-1)}{cd}\|\partial_{x_1}(f_m- f_{m'})\|_{L^2}\\
		&\geq \frac 1{8d}\|\partial_{x_1}(f_m- f_{m'})\|_{L^2}.
		\end{split}
		\end{equation*}
		Moreover, as the first partial derivatives of the $\Psi_{j,r}$ still have disjoint support, they are orthonormal in $L^2$ and  by Parseval's identity we have
		\begin{equation}\label{I-est-2}
		\begin{split}
		\|\partial_{x_1}(f_m- f_{m'})\|_{L^2}^2&=\kappa^22^{-2j(\alpha+d/2)}\sum_{j=1}^{n_j}|\beta_{r,m}-\beta_{r,m'}|^2\|\partial_{x_1}\Psi_{j,r}\|_{L^2}^2\\
		&=\|\partial_{x_1}\Psi_{0,1}\|_{L^2}^2\kappa^22^{-2j(\alpha-1+d/2)}\sum_{j=1}^{n_j}|\beta_{r,m}-\beta_{r,m'}|^2.
		\end{split}
		\end{equation}
		By the Varshamov-Gilbert-bound (Example 3.1.4 in \cite{nicklgine}), for constants $c_1,c_2>0$ independent of $j$, we can find a subset $\mathcal M_j\subset \{-1,1\}^{c_02^{jd}}$ of cardinality $M_j=2^{c_12^{jd}}$ such that 
		$$\sum_{j=1}^{n_j}|\beta_{r,m}-\beta_{r,m'}|^2\geq c_22^{jd}$$
		whenever $m\neq m'$. For such a subset $\mathcal M_j$, by (\ref{I-est-2}) we have
		\begin{equation}\label{lowerbound2}
		I\gtrsim\|\partial_{x_1}(f_m- f_{m'})\|_{L^2}\gtrsim 2^{-j(\alpha-1)}.
		\end{equation}
		\par
		6. We next show that $II$ and $III$ in (\ref{lowerbound}) are of smaller order as $j \to \infty$. With the above choice of $f_m$'s, we have from Parseval's identity and (\ref{c-h-mult})
		\begin{equation*}
		\begin{split}
		II^2&\leq \|f_m-f_{m'}\|_{L^2}^2\|u_{f_m}\|^2_{\mathcal C^2}= \kappa^22^{-2j(\alpha+d/2)}\sum_{r=1}^{n_j}|\beta_{r,m}-\beta_{r,m'}|^2 \|u_{f_m}\|^2_{\mathcal C^2}\\
		&\lesssim 2^{-2j\alpha}=o (2^{-2j(\alpha-1)}),
		\end{split}
		\end{equation*}
		and for term $III$ we have, by (\ref{fmfm'}), (\ref{h-int}), Lemma \ref{lem-div-h2} and arguing as in the first display of Step 7 to follow, that
		\begin{align*}
		\|u_{f_m}- & u_{f_{m'}}\|_{H^1}  \lesssim \|u_{f_m}-u_{f_m'}\|^{1/2}_{H^2}  \|u_{f_m}-u_{f_{m'}}\|^{1/2}_{L^2} \\
		&\lesssim \|\nabla\cdot ((f_m-f_{m'})\nabla u_{f_{m'}})]\|^{1/2}_{L^2}  \|[\nabla\cdot ((f_m-f_{m'})\nabla u_{f_{m'}})]\|^{1/2}_{(H^2_0)^*} \\
		& \lesssim \|f_m-f_{m'}\|_{H^1}^{1/2} \|f_m-f_{m'}\|^{1/2}_{H^{-1}} \lesssim 2^{-j\alpha}=o (2^{-j(\alpha-1)}),
		\end{align*}
		where the first factor in the last line is bounded by $2^{-j(\alpha/2-1/2)}$ by similar arguments as in (\ref{I-est-2}). Combining the last two displayed estimates with (\ref{lowerbound}) and (\ref{lowerbound2}) gives the overall lower bound 
		\[\|u_{f_m}-u_{f_{m'}}\|_{H^2} \gtrsim 2^{-j(\alpha-1)} \approx \varepsilon^{\frac{2(\alpha-1)}{2(\alpha+1)+d}}\] with choice $j=j_\varepsilon$ such that $2^{j}\simeq \varepsilon^{-2/(2\alpha+2+d)}$.
		\par
		7. Now we show the upper bound. Arguing as in (\ref{obpertu}), using Lemma \ref{lem-div-h2}, integrating by parts and using the wavelet characterisation of the $H^{-1}(\mathbb R^d)$-norm (e.g., Section 4.3 in \cite{nicklgine} with $B^s_{2,2}=H^s, s \in \mathbb R$) as well as the interior support of the $\Psi_{j,r}$, we estimate
		\begin{equation*}
		\begin{split}
		\|u_{f_m}-u_{f_0}\|_{L^2}^2&\lesssim \left\|\nabla \cdot \left((f_m-f_{0})\nabla u_{f_0}\right)\right\|_{(H^2_0)^*}^2\\
		&=\left(\sup_{\|\psi\|_{H^2_0}\leq 1}\left|\int_{\mathbb R^d} \nabla \psi\cdot \nabla u_{f_0} (f_m-f_{0})\right|\right)^2\\
		&\lesssim \|f_m-{f_{0}}\|_{H^{-1}(\mathbb{R}^d)}^2 \|u_{f_0}\|_{\mathcal C^1} \\
		&\simeq \kappa^22^{-2j(\alpha+d/2+1)}\sum_{r=1}^{n_j}1 \lesssim 2^{-2j(\alpha+1)}.
		\end{split}
		\end{equation*}
		By definition of $M_j$, using the results in Section 7.4 in \cite{n17} and arguing as in (6.16) in \cite{nicklgine} we thus bound the information distances as
		$$\textnormal{KL}(\mathbb P^{\varepsilon}_{u_{f_m}},\mathbb P^{\varepsilon}_{u_{f_0}})\lesssim \varepsilon^{-2} \|u_{f_m} - u_{f_0}\|_{L^2}^2 \lesssim  \varepsilon^{-2}2^{-2j(\alpha+1)}=2^{jd}\lesssim \log M_j,$$
		so that the overall result now follows from Theorem 6.3.2 in \cite{nicklgine}.
	\end{proof}

	\subsection{Proof of Theorems \ref{thm-schr-pred} and \ref{thm-schr-f}}\label{sec-schr-pfs}
	The proof of Theorem \ref{thm-schr-pred} follows the same principle as the proof of Theorem \ref{thm-div-pred}. By arguing exactly as in the first two steps of the proof of Theorem \ref{thm-div-pred}, in order to be able to apply Theorem \ref{thm-gen}, we now verify that the map 
	\[\mathscr G:H^\alpha_c\to L^2, \qquad \mathscr G(F):=G(\Phi\circ F), \]
	satisfies (\ref{entrcond}) with $\mathbb H = L^2, \gamma =4, \kappa=2$. Let $F,H\in H^\alpha$ and $f=\Phi\circ F$, $h=\Phi\circ H\in\mathcal F$. By (\ref{schroedeq}), $u_f-u_h$ satisfies 
	\[(u_f-u_h)|_{\partial\mathcal O}=0, \qquad L_f[u_f-u_h]=(L_h-L_f)[u_h]=(f-h)u_h \] where $L_f$ is defined in Section \ref{sec-schr-G} below. Using this, the norm estimate (\ref{schr-h-2l2}), Lemma \ref{lem-schr-c2}, the embedding $H^\alpha\subseteq C^2(\mathcal O)$ as well as (\ref{phi-neg-sob}), we can then estimate
	\begin{equation*}
	\begin{split}
	\|\mathscr G(F)-&\mathscr G(H) \|_{L^2}=\left\|u_f-u_h \right\|_{L^2}\\
	&\lesssim \left(1+\|f\|_{\infty}\right)\left\|(f-h)u_h \right\|_{(H^2_0)^*}\\
	&\leq \left(1+\|f\|_{\infty}\right)\|u_h\|_{\mathcal C^2}\left\|f-h\right\|_{(H^2_0)^*}\\
	&\lesssim  \left(1+\|f\|_{\infty}\right)\left(1+\|h\|_{\infty}\right)\left\|f-h\right\|_{(H^{2})^*}\\
	&\lesssim \left(1+\|F\|_{\infty}^2\vee\|H\|_{\infty}^2\right)\left\|F-H \right\|_{(H^{2})^*}\left(1+ \|F\|_{C^2}^2\vee \|H\|_{C^2}^2\right)  \\
	&\lesssim \left(1+\|F\|_{H^\alpha}^4\vee \|H\|_{H^\alpha}^4\right)\left\|F-H \right\|_{(H^{2})^*}.
	\end{split}
	\end{equation*}
	Thus (\ref{entrcond}) is fulfilled for $\gamma =4$ and $\kappa=2$. The existence of maximizers now follows from the first part of Theorem \ref{thm-gen}. The proof of the concentration inequality (\ref{schr-conc}) is completely analogous to the proof of (\ref{div-conc}), and the convergence rate (\ref{schr-rate}) follows from the same argument as in the proof of Theorem \ref{thm-div-pred}, utilizing Lemma \ref{lem-schr-c2} in place of Lemma \ref{lem-div-bd}. 
	
	Finally, the proof Theorem \ref{thm-schr-f} is analogous to that of Theorem \ref{thm-div-f}, but using Lemma \ref{lem-schr-stab} instead of Lemma \ref{lem-div-stab}, and is left to the reader.

	\section{Some PDE facts}\label{sec-pde-facts}

	In this section, we collect some key PDE facts which are needed to prove the results in Section \ref{sec-pderes}. 
	
	\subsection{Preliminaries}
	Besides the classical H\"older spaces $C^\alpha(\mathcal O)$, we will also need the H\"older-Zygmund spaces $\mathcal C^\alpha(\mathcal O)$, see Section 3.4.2 in \cite{triebel} for definitions. For $\alpha\geq 0$, $\alpha\notin\mathbb N$, we have that $C^\alpha=\mathcal C^\alpha$ with equivalent norms, and we have the continuous embeddings $\mathcal C^{\alpha'}\subseteq C^\alpha\subseteq \mathcal C^\alpha$ for all $\alpha'>\alpha\geq 0$.
	\par
	We will repeatedly use the multiplicative inequalities
	\begin{align}
	\|fg\|_{H^\alpha}&\lesssim \|f\|_{H^\alpha}\|g\|_{H^\alpha},\qquad \alpha>d/2, \label{h-mult}\\
	\|fg\|_{H^\alpha}&\lesssim \|f\|_{\mathcal C^\alpha}\|g\|_{H^\alpha}, \qquad \alpha\geq 0\label{c-h-mult},\\
	\|fg\|_{\mathcal C^\alpha}&\lesssim\|f\|_{\mathcal C^\alpha}\|g\|_{\mathcal C^\alpha}, \qquad \alpha\geq 0 \label{c-mult}
	\end{align}
	for all $f,g$ in the appropriate function spaces, which follow from Remark 1 on p.143 and Theorem 2.8.3 in \cite{triebel}. For any $\alpha>d/2$ and $0\leq \eta<\alpha-d/2$, we also need the continuous embedding $H^\alpha\subseteq C^\eta$, with the norm estimate
	\begin{equation}\label{h-emb}
	\forall f\in H^\alpha,\;\;	\|f\|_{C^\eta}\lesssim\|f\|_{H^\alpha}.
	\end{equation}
	
	Let $\text{tr}[\cdot]$ denote the usual trace operator for functions defined on $\mathcal O$ (for the definition on Sobolev spaces, see, e.g., Chapter 5.5 in \cite{evans}). In this and the next section, we will repeatedly use the fact that the standard Laplacian $\Delta$ and $\text{tr}[\cdot]$ establish topological isomorphisms between appropriate Sobolev and H\"older-Zygmund spaces. That is, for each $\alpha\geq 0$, we have the topological isomorphisms
	\begin{align}
	&(\Delta,\textnormal{tr}): H^{\alpha+2}(\mathcal O)\to H^\alpha(\mathcal O)\times H^{\alpha+3/2}(\partial\mathcal O), \quad u\mapsto (\Delta u, \textnormal{tr}[u]),\label{h-isom} \\
	&(\Delta,\textnormal{tr}): \mathcal C^{\alpha+2}(\mathcal O)\to\mathcal C^{\alpha}(\mathcal O)\times \mathcal C^{\alpha+2}(\partial \mathcal O), \quad u\mapsto (\Delta u, \textnormal{tr}[u]),\label{c-isom}
	\end{align}
	which follow from Theorem II.5.4 in \cite{lionsmagenes} and Theorem 4.3.4 in \cite{triebel} respectively. Moreover, for any $\alpha\geq 1$, we will use the notation
	\begin{align}\label{h0}
	H^\alpha_0(\mathcal O):=\left\{f\in H^\alpha(\mathcal O)\;\middle| \; \textnormal{tr}[f]=0 \right\},\quad \mathcal C^\alpha_0(\mathcal O):=\left\{f\in \mathcal C^\alpha(\mathcal O)\;\middle| \; \textnormal{tr}[f]=0 \right\}.
	\end{align}
	We also need the following interpolation inequalities. For all $\beta_1,\beta_2\geq 0$ and $\theta\in [0,1]$, there exists a constant $C<\infty$ such that 
	\begin{align}
	&\forall u\in \mathcal C^{\beta_1}\cap \mathcal C^{\beta_2}:\quad \|u\|_{\mathcal C^{\theta \beta_1+(1-\theta)\beta_2}}\leq C \|u\|_{\mathcal C^{\beta_1}}^\theta\|u\|_{\mathcal C^{\beta_2}}^{1-\theta}, \label{c-int}\\
	&\forall u\in H^{\beta_1}\cap H^{\beta_2}:\quad \|u\|_{H^{\theta \beta_1+(1-\theta)\beta_2}}\leq C \|u\|_{H^{\beta_1}}^\theta\|u\|_{H^{\beta_2}}^{1-\theta}, \label{h-int}
	\end{align}
	see Theorems 1.3.3 and 4.3.1 in \cite{T78} (and note $\mathcal C^\beta=B^\beta_{\infty, \infty}, H^\beta=B^\beta_{2, 2}$).
	\par

	\subsection{Divergence form equation}\label{sec-div-facts}

	\subsubsection{Estimates for $V_f$}\label{vfop}
	For each $f\in C^1(\bar{\mathcal O})$ with $f\geq K_{min}>0$, we define the differential operator
	$$L_f:H^2_0(\mathcal O)\to L^2(\mathcal O),\qquad L_f[u]=	\nabla\cdot (f\nabla u).$$
	By standard theory for elliptic PDEs, $L_f$ has a linear, continuous inverse operator, which we denote by
	$$V_f:L^2(\mathcal O)\to H^2_0(\mathcal O),\qquad \psi\mapsto V_f\left[\psi\right],$$
	see \cite{evans}, Theorem 4 in Chapter 6.3. In other words, for each right hand side $\psi\in L^2$, there exists a unique function $w_{f,\psi}:=V_f\left[\psi\right]\in H^2_0$ solving the Dirichlet problem 
	\begin{equation}\label{div2}
	\begin{cases}
	L_f[w_{f,\psi}]=\psi \quad \textnormal{on }\mathcal O,\\
	w_{f,\psi}=0 \quad \textnormal{on }\partial \mathcal O
	\end{cases}
	\end{equation}
	weakly, i.e. in the sense that the identity
	\begin{equation}\label{div-wk}
	-\int_{\mathcal O}\sum_{i=1}^{d}f D_iw_{f,\psi}D_iv \;=\int_{\mathcal O}\psi v \; 
	\end{equation}
	holds for all test functions $v\in H^1_0(\mathcal O)$ (cf. \cite{evans}, Chapter 6). By the zero boundary conditions of (\ref{div}) and the divergence theorem, any classical solution (i.e. $C^2$ solution) must be equal to the unique weak solution when interpreted as an $H^2_0$ function.
	\par 
	Theorem 4 in Chapter 6.3 of \cite{evans} implies that there exists a constant $C=C_f$ (allowed to depend on $f$) such that for all $\psi\in L^2$, we have the norm estimate
	$\|V_f\left[\psi\right]\|_{H^2}\leq C_f\|\psi\|_{L^2},$ and we need a result that tracks the dependence of $C_f$ on $f$ in a quantitative way. We first establish that when we only seek an $L^p\to L^p$-estimate, $p\in\{2,\infty\}$, rather than an $L^2 \to H^2$-estimate, the constant merely depends on the lower bound $K_{min}$ for $f$.

	\begin{lem}\label{div-lem-lp}
		Let $K_{min}>0$. Then there exists $C=C(d,\mathcal O, K_{min})$ such that for all $f\in C^2(\mathcal O)$ with $f\geq K_{min}>0$ and $\psi\in L^2$, we have
		\begin{equation}\label{div-l2}
		\|V_f\left[\psi\right]\|_{L^2}\leq C\|\psi\|_{L^2} 
		\end{equation}
		and for all $\psi\in C^\eta(\mathcal O), \eta>0$,
		\begin{equation}\label{div-linf}
		\|V_f\left[\psi\right]\|_{\infty}\leq C\|\psi\|_{\infty}.
		\end{equation}
	\end{lem}
	\begin{proof}
		Assume first that $\psi \in C^\eta(\mathcal O)$ so that $V_f[\psi]\in C(\bar{\mathcal O})\cap C^2(\mathcal O)$ (see after (\ref{div})). Then we have the Feynman-Kac formula
		\begin{equation}\label{div-fk}
		V_f[\psi](x)=-\frac{1}{2}\E^x\left[\int_0^{\tau_\mathcal O}\psi(X^f_s)ds \right], \qquad x\in\mathcal O,
		\end{equation}
		where $(X^f_s:s \ge 0)$ is a diffusion Markov process started at $x \in \mathcal O$ with infinitesimal generator $L_f/2$ and expectation operator $\E^x$, and where $\tau_\mathcal O$ is the exit time of $X^f_s$ from $\mathcal O$, see, e.g., Theorem 1.2 in Section II of \cite{bass}. We also record that, by Theorem 4.3 in Section VII of \cite{bass} and inspection of its proof, there exists a constant $c_1$ only depending on the lower bound $K_{min}<f$ and on $d$, such that the transition densities of $(X^f_s: s \ge 0)$ exist and satisfy the estimate 
		\begin{equation}\label{div-heat}
		p_f(t,x,y) \le c_1 t^{-d/2}, \qquad t>0,\quad x,y\in\R^d. 
		\end{equation}
		Then, arguing as in the proof of Theorem 1.17 in \cite{chungzhao}, with (\ref{div-heat}) replacing the standard heat kernel estimate for Brownian motion, we obtain that $\sup_{x \in \mathcal O}\E^x \tau_\mathcal O\leq c,$ with $c=c(\mathcal O, d, c_1)$, and hence (\ref{div-linf}) follows from
		\begin{equation}\label{vf1}
		\|V_f[1]\|_\infty \le \sup_{x \in \mathcal O}\E^x \tau_\mathcal O \le c.
		\end{equation}
		Using what precedes one further shows that $V_f$ has a representation via a non-negative and symmetric integral kernel $G_f(\cdot,\cdot)$, such that
		\begin{equation}\label{div-green}
		V_f[\psi](x)=-\int_{\mathcal O}G_f(x,y)\psi(y)dy, ~ x\in\mathcal O,~~\forall~ \psi\in C^\eta(\mathcal O).
		\end{equation}
		Then using (\ref{vf1}), the Cauchy-Schwarz inequality and the positivity of $G$ we have for all $\psi\in C^\eta(\mathcal O)$,
		\[\|V_f[\psi]\|_{L^2}^2 \le \int_{\mathcal O} \int_{\mathcal O} G_f(x,y)dy \int_{\mathcal O} G_f(x,y)\psi^2(y)dy dx \le \|V_f[1]\|_\infty^2 \|\psi\|_{L^2}^2,\]
		whence (\ref{div-l2}) follows for $\psi\in C^\eta(\mathcal O)$, and extends to $\psi\in L^2$ by approximation since $V_f$ is a continuous operator on $L^2(\mathcal O)$ (as established above).
	\end{proof}
	
	Lemma \ref{div-lem-lp} will be used in the proof of the following stronger elliptic regularity estimate.

	\begin{lem}\label{lem-div-h2}
		Let $K_{min}>0$. Then there exists a universal constant $C>0$ such that for all $f\in C^{2}(\mathcal O)$ with $f\geq K_{min}$ and $\psi\in L^2(\mathcal O)$, the unique weak solution $w_{f,\psi}=V_f[\psi]$ to (\ref{div2}) satisfies
		\begin{align}
		\|V_f\left[\psi\right]\|_{H^2}&\leq C\left(1+\|f\|_{C^1}\right)\|\psi\|_{L^2}, \label{div-l2h2}\\
		\|V_f\left[\psi\right]\|_{L^2}&\leq C\left(1+\|f\|_{C^1}\right)\|\psi\|_{(H^2_0)^*} \label{div-h-2},
		\end{align}
		where $C$ only depends on $K_{min}$ and $\mathcal O, d$.
	\end{lem}

	\begin{proof}
		Let $f\in C^1$ and $\psi\in L^2$. By (\ref{h-isom}), there exists a constant $C>0$ depending only on $\mathcal O, d$ such that for all $u\in H^2_0$,
		\begin{equation}\label{lapl-isom}
		C^{-1}\|\Delta u\|_{L^2}\leq \|u\|_{H^2}\leq C\|\Delta u\|_{L^2}.
		\end{equation}
		Moreover we have by the definition of $L_f$ that
		\begin{equation}\label{div-lapl}
		\Delta u=f^{-1}(L_fu-\nabla f\cdot \nabla u).
		\end{equation}
		Writing $w=w_{f,\psi}$ and utilising (\ref{lapl-isom}) and (\ref{div-lapl}), we can estimate
		\begin{equation}\label{h2start}
		\begin{split}
		\|w\|_{H^2}&\leq C\|\Delta w\|_{L^2}=C\left\|f^{-1}(\psi-\nabla w\cdot\nabla f)\right\|_{L^2}\\
		&\leq CK_{min}^{-1}\left(\|\psi\|_{L^2}+\|f\|_{C_1}\|w\|_{H^1}.\right)
		\end{split}
		\end{equation}
		By choosing the test function $-w\in H^1_0$ in the weak formulation (\ref{div-wk}), we have that
		\[K_{min}\int_{\mathcal O}|Dw|^2\leq \int_{\mathcal O}\sum_{i=1}^{d}f(D_iw)^2 =\int_{\mathcal O}-\psi w \;\leq \frac 12 \int_{\mathcal O}(\psi^2+w^2) \;. \]
		Combining this with (\ref{h2start}) and Lemma \ref{div-lem-lp}, we finally obtain that for constants $C',C'',C'''$ only depending on $K_{min}$ and $\mathcal O$, we have
		\begin{equation*}
		\begin{split}
		\|w\|_{H^2}&\leq C'K_{min}^{-1}\left(\|\psi\|_{L^2}+\|f\|_{C^1}C''(\|\psi\|_{L^2}+\|w\|_{L^2} )\right)\\
		&=C'''\left(1+\|f\|_{C^1}\right)\|\psi\|_{L^2},
		\end{split}
		\end{equation*}
		which proves (\ref{div-l2h2}).
		\par 
		Next, using the divergence theorem and (\ref{div-l2h2}), we obtain (\ref{div-h-2}) from
		\begin{align*}
		\|V_{f}[\psi]\|_{L^2}&=\sup_{\varphi\in C^\infty_c,\;\; \|\varphi\|_{L^2}\leq 1}\left|\int_{\mathcal O} V_{f}[\psi]\varphi \right|\\
		&=\sup_{\varphi\in C_c^\infty (\mathcal O),\;\; \|\varphi\|_{L^2}\leq 1}\left|\int_{\mathcal O} V_{f}[\psi]L_fV_f[\varphi] \right|\\
		&=\sup_{\varphi\in C_c^\infty (\mathcal O),\;\; \|\varphi\|_{L^2}\leq 1}\left|\int_{\mathcal O} \psi V_f[\varphi] \right|\\
		& \leq C(1+\|f\|_{C^1})\sup_{\varphi \in H^2_0,\;\; \|\varphi\|_{H^2}\leq 1}\left|\int_{\mathcal O} \psi \varphi \right| =C(1+\|f\|_{C^1})\|\psi\|_{(H^2_0)^*}.
		\end{align*}
	\end{proof}

	\subsubsection{Estimates for $G$}\label{sec-div-G}
	Now we turn to the forward map $G$ representing the solutions of the PDE (\ref{div}). The following norm estimate for the $\mathcal C^2$-H\"older-Zygmund norm of $G(f)=u_f$ is needed.
	
	\begin{lem}\label{lem-div-c}
		Suppose that for some $K_{min}>0$, $\alpha>d/2+2$ and $g\in \mathcal C^\eta(\mathcal O), \eta>0$, $\tilde{\mathcal F}$ is as in (\ref{ftilde}) and $u_f$ denotes the unique solution of (\ref{div}). Then there exists $C=C(d, \mathcal O, K_{min}, \|g\|_\infty)$ such that for all $f\in\tilde{\mathcal F}$,
		\begin{equation}\label{div-c}
		\|u_f\|_{\mathcal C^2}\leq C\left(1+\|f\|_{\mathcal C^1}^2\right).
		\end{equation}
	\end{lem}
	\begin{proof}
		The proof is similar to that of Lemma \ref{lem-div-h2}. By (\ref{c-isom}), there exists a constant $C>0$ depending only on $\mathcal O, d$ such that for all functions $u\in \mathcal C^2_0(\mathcal O)$, we have
		\begin{equation}\label{c2-isom}
		C^{-1}\|\Delta u\|_{\mathcal C^0}\leq \|u\|_{\mathcal C^2}\leq C \|\Delta u\|_{\mathcal C^0}.
		\end{equation}
		Using this, the PDE (\ref{div}), the multiplicative inequality (\ref{c-mult}) and the interpolation inequality (\ref{c-int}), we can estimate as in (\ref{h2start})
		\begin{equation*}
		\begin{split}
		\|u_{f}\|_{\mathcal C^2}&\lesssim \|f^{-1}(g-\nabla f\cdot \nabla u_f)\|_{\mathcal C^0} \lesssim \|f^{-1}\|_{\mathcal C^0}\left(\|g\|_{\mathcal C^0}+\|f\|_{\mathcal C^1} \|u_f\|_{\mathcal C^1}\right)\\
		&\lesssim K_{min}^{-1}\left(\|g\|_{\mathcal C^0}+\|f\|_{\mathcal C^1}\|u_f\|_{\mathcal C^2}^{1/2}\|u_f\|_{\mathcal C^0}^{1/2}\right).
		\end{split}
		\end{equation*}
		Dividing this inequality by $\|u_f\|_{\mathcal C^2}^{1/2}$ whenever $\|u_f\|_{\mathcal C^2}^{1/2}\geq 1$ and otherwise estimating it by $1$, we obtain that 
		\begin{equation*}
		\|u_{f}\|_{\mathcal C^2}\lesssim 1+ K_{min}^{-2}\left(\|g\|_{\mathcal C^0}^2+\|f\|_{\mathcal C^1}^2\|u_f\|_{\mathcal C_0}\right)\lesssim 1+ K_{min}^{-2}\left(\|g\|_{\infty}^2+\|f\|_{\mathcal C^1}^2\|g\|_{\infty}\right)
		\end{equation*}
		where in last step we used $\|\cdot\|_{\mathcal C^0}\lesssim\|\cdot\|_{\infty}$ and Lemma \ref{div-lem-lp}.
	\end{proof}

	We also need that the forward map $G$ maps bounded sets in $H^\alpha$ onto bounded sets in $H^{\alpha+1}$.
	\begin{lem}\label{lem-div-bd}
		Suppose that $\alpha,\tilde{\mathcal F}$ are as in Lemma \ref{lem-div-c} and for some $g\in H^{\alpha-1}(\mathcal O)$, let $u_f=w_{f,g}, f \in \tilde {\mathcal F},$ be the unique solution of (\ref{div}). Then $u_f \in H^{\alpha+1}(\mathcal O)$ and there exists a constant $C=C(\alpha, d, \mathcal O, K_{min})>0$ such that
		\begin{equation}\label{div-h-est}
		\|u_{f}\|_{H^{\alpha+1}}\leq C\big(1+\|f\|_{H^\alpha}^{\alpha^2+\alpha}\big)\big( \|g\|_{H^{\alpha-1}}^{\alpha+1} \vee \|g\|_{H^{\alpha-1}}^{1/(\alpha+1)} \big).
		\end{equation}
	\end{lem}
	
	\begin{proof}
		First, suppose $f\in C^\infty\cap \tilde{\mathcal F}$. %By a standard regularity result for elliptic PDEs (see \cite{evans}, Theorem 6 in Chapter 6.3), we then have that $u_f\in C^\infty$. 
		By (\ref{h-isom}), the standard Laplacian $\Delta$ establishes an isomorphism between $H^{\alpha+1}_0$ and $H^{\alpha-1}$, and by Theorem 8.13 in \cite{gt}, $u_f \in H^{\alpha+1}_0$. Then (\ref{div-lapl}) and the multiplicative inequality (\ref{h-mult}) give
		\begin{equation*}
		\begin{split}
		\|u_f\|_{H^{\alpha+1}}&\lesssim \|f^{-1}(g-\nabla f\cdot \nabla u_f)\|_{H^{\alpha-1}}\\
		& \lesssim\|f^{-1}\|_{H^{\alpha-1}}(\|g\|_{H^{\alpha-1}}+\|f\|_{H^\alpha}\|u_f\|_{H^\alpha}).
		\end{split}
		\end{equation*}
		Noting that the map $\Psi: (K_{min},\infty)\to\mathbb R,~ x\mapsto x^{-1}$ satisfies (\ref{phibddder}), (\ref{phi-sob}) implies that there exists $c>0$ such that for all $f\in\mathcal F$,
		\[\|f^{-1}\|_{H^{\alpha-1}}\leq c(1+\|f\|_{H^{\alpha-1}}^{\alpha-1}).\]
		Using this and (\ref{h-int}), we obtain
		\begin{equation*}
		\begin{split}
		\|u_f\|_{H^{\alpha+1}}&\lesssim(1+ \|f\|_{H^{\alpha-1}}^{\alpha-1})(\|g\|_{H^{\alpha-1}}+\|f\|_{H^\alpha}\|u_f\|_{H^\alpha})\\
		& \lesssim(1+ \|f\|_{H^{\alpha}}^\alpha)\big(\|g\|_{H^{\alpha-1}}+\|u_f\|_{H^{\alpha+1}}^{\frac \alpha{\alpha+1}}\|u_f\|_{L^2}^{\frac{1}{\alpha+1}}\big)
		\end{split}
		\end{equation*}
		When $\|u_f\|_{H^{\alpha+1}}\leq 1$ we use (\ref{div-l2}) to deduce
		\begin{equation*} 
		\|u_f\|_{H^{\alpha+1}} \lesssim (1+ \|f\|_{H^{\alpha}}^\alpha)\big(\|g\|_{H^{\alpha-1}}+\|g\|_{L^2}^{\frac{1}{\alpha+1}}\big),
		\end{equation*}
		and when $\|u_f\|_{H^{\alpha+1}}\geq 1$, then dividing both sides by $\|u_f\|_{H^{\alpha+1}}^{\frac \alpha{\alpha+1}}$ and using again (\ref{div-l2}) yields
		\begin{equation*}
		\|u_f\|^{1/(\alpha+1)}_{H^{\alpha+1}}\lesssim (1+ \|f\|_{H^{\alpha}}^\alpha)\big(\|g\|_{H^{\alpha-1}}+\|g\|_{L^2}^{\frac{1}{\alpha+1}}\big).
		\end{equation*}
		Combining the preceding bounds and using $\|\cdot\|_{L^2} \lesssim \|\cdot\|_{H^{\alpha-1}}$ implies (\ref{div-h-est}) for smooth $f \in  \tilde{\mathcal F}$. Now for any $f\in\tilde{\mathcal F}$, take $f_n\in C^\infty(\mathcal O), f_n>K_{min}/2,$  such that $f_n\to f$ in $H^{\alpha}$ as $n \to \infty$, and hence by (\ref{div-h-est}) the sequence $u_{f_n}$ is bounded in $H^{\alpha+1}$. Then applying (\ref{obpertu}) to $u_{f_n}-u_{f_m}, m,n \in \mathbb N,$ and applying (\ref{div-h-est}) with $g=\nabla \cdot ((f_m- f_n) \nabla u_{f_m}),$ one shows that $u_{f_n}$ is a Cauchy sequence in $H^{\alpha+1}$ converging to $u_f$, and taking limits extends the inequality (\ref{div-h-est}) to the general case $f\in\mathcal F$.
	\end{proof}

	\subsubsection{Stability Estimates for $G^{-1}$}
	
	The following estimate for the inverse map $u_f \mapsto f$ allows to obtain convergence rates for $\|\hat f-f_0\|_{L^2}$ via rates for $\|u_{\hat f} - u_{f_0}\|_{H^2}$, with choices $f_0 = f_1$ and $\hat f=f_2$. As $\hat f$ is random we explicitly track the dependence of the constants on $f_2$.
	
	\begin{lem}\label{lem-div-stab}
		Let $\alpha>d/2+2, g_{min}, K_{min}, B, \eta$ be given, positive constants and let $\tilde{\mathcal F}$ be given by (\ref{ftilde}). For $g\in C^\eta(\mathcal O)$ with $\inf_{x \in \mathcal O}g(x)\geq g_{min}$, denote by $u_f$ the unique solution of (\ref{div}). Then there exists $C=C(g_{min}, K_{min}, B, \mathcal O,d)<\infty$ such that for all $f_1,f_2\in \tilde{\mathcal F}$ with $\|f_1\|_{C^1} \vee \|u_{f_1}\|_{C^2} \le B$, we have
		\begin{equation*}
		\|f_1-f_2\|_{L^2}\leq C\|f_2\|_{C_1}\|u_{f_1}-u_{f_2}\|_{H^2}.
		\end{equation*}
	\end{lem}
	\begin{proof}
		For $f_1,f_2\in\tilde{\mathcal F}$  write $h=f_1-f_2$. By (\ref{div}), we have
		\begin{equation}\label{div-diff}
		\begin{split}
		\nabla \cdot (h \nabla u_{f_1}) &= \nabla\cdot  (f_1 \nabla u_{f_1}) - \nabla\cdot (f_2 \nabla u_{f_2}) - \nabla\cdot  (f_2 \nabla (u_{f_1}-u_{f_2})) \\
		&=\nabla \cdot (f_2 \nabla (u_{f_2}-u_{f_1})). 
		\end{split}
		\end{equation}
		We can upper bound the $\|\cdot\|_{L^2}$-norm of the right hand side by
		\begin{align}
		\|\nabla \cdot (f_2 \nabla (u_{f_2}-u_{f_1}))\|_{L^2}&\leq \|\nabla f_2\|_{\infty}\|u_{f_2}-u_{f_1}\|_{H^1} +\|f_2\|_{\infty}\|u_{f_2}-u_{f_1}\|_{H^2}\notag \\
		&\leq 2\|f_2\|_{C^1}\|u_{f_2}-u_{f_1}\|_{H^2}.\label{diff-ub}
		\end{align}
		Next, we lower bound the $\|\cdot\|_{L^2}$-norm of the left side of (\ref{div-diff}). For regular enough $v$ we see from Green's identity (p.17 in \cite{gt}) that
		\begin{align*}
		\langle \Delta u_{f_1}, v^2 \rangle_{L^2} + \frac{1}{2} \langle \nabla u_{f_1}, \nabla (v^2) \rangle_{L^2} = \frac{1}{2}\langle \Delta u_{f_1}, v^2 \rangle_{L^2} + \frac{1}{2}\int_{\partial \mathcal O}  \frac{\partial u_{f_1}}{\partial n} v^2.
		\end{align*}
		Moreover for $v=e^{-\lambda u_{f_1}}h$ with $\lambda>0$ to be chosen we have
		\begin{equation*}
		\frac{1}{2} \int_{\mathcal O}  \nabla (v^2) \cdot \nabla u_{f_1} =- \int_{\mathcal O} \lambda \|\nabla u_{f_1}\|^2 v^2 + \int_\mathcal O v e^{-\lambda u_{f_1}}\nabla h \cdot \nabla u_{f_1},
		\end{equation*}
		so that by the Cauchy-Schwarz inequality
		\begin{align} \label{keylb}
		& \left|\int_{\mathcal O}\Big(\frac{1}{2}\Delta u_{f_1}+\lambda \|\nabla u_{f_1}\|^2\Big)v^2 + \int_{\partial \mathcal O}\frac 12 \frac{\partial u_{f_1}}{\partial n}v^2\right| \notag \\
		& = \left|\langle (\Delta u_{f_1} + \lambda \|\nabla u_{f_1}\|^2), v^2 \rangle_{L^2} + \frac{1}{2} \langle \nabla u_{f_1}, \nabla (v^2)\rangle_{L^2}\right| \notag \\
		&= \left|\langle h \Delta u_{f_1} + \nabla h \cdot \nabla u_{f_1}, h e^{-2\lambda u_{f_1}} \rangle_{L^2} \right| \le \mu \|\nabla \cdot (h \nabla u_{f_1})\|_{L^2}  \|h\|_{L^2}
		\end{align}
		for $\mu=\exp(2\lambda \|u_{f_1}\|_{\infty})$. [The preceding argument is adapted from the proof of Theorem 4.1 in \cite{itokunisch}.] We next lower bound the multipliers of $v^2$ in the integrands in the first line of the last display. First we have
		\begin{equation*}
		0<g_{min} \le g= L_{f_1}u_{f_1} = f_1(x)  \Delta u_{f_1} + \nabla f_1 \cdot \nabla u_{f_1},~\text{ on } \mathcal O,
		\end{equation*}
		so that either $\Delta u_{f_1}(x) \ge g_{min}/2\|f_1\|_\infty$ or $\|\nabla u_{f_1}(x)\|^2 \ge (g_{min} / 2\|f_1\|_{C^1})^2$ on $\mathcal O$. Since $\|\Delta u_{f_1}\|_\infty\le c(B)$ this implies for $\lambda=\lambda (g_{min}, B)$ large enough that 
		\begin{equation}\label{richterlb}
		\frac{1}{2}\Delta u_{f_1}(x) + \lambda \|\nabla u_{f_1}(x)\|^2 \ge c_0>0,~~ x \in \mathcal O,
		\end{equation}
		for some $c_0=c_0(g_{min}, B)$. Next, for the integral over $\partial \mathcal O$, we use again $L_{f_1}u_{f_1}=g> 0$ and apply the Hopf boundary point Lemma 6.4.2 in \cite{evans}: We have $u_{f_1}(x_0)=0$ for any $x_0 \in \partial \mathcal O$ but $u_{f_1}(x)<0$ for all $x \in \mathcal O$: Indeed, by $g \ge g_{min}>0$ and the Feynman-Kac formula (\ref{div-fk}) (with $g=\psi$), it suffices to lower bound $\mathbb E^x \tau_\mathcal O$ which satisfies, by Markov's inequality $$\mathbb E^x \tau_\mathcal O \ge \mathbb P^x(\tau_\mathcal O >1) \ge \mathbb P^x\big(\sup_{0 <s\le 1}\|X_s-x\|<\|x-\partial \mathcal O\|\big)>0$$ in view Theorem V.2.5 in \cite{bass} with $\psi(s)=x$ identically for all $s$. Lemma 6.4.2 in \cite{evans} now gives $\partial u_{f_1}/\partial n \ge 0$ for all $x\in \partial \mathcal O.$ Combining this with (\ref{keylb}) and (\ref{richterlb}) we deduce
		\begin{equation*}
		\|\nabla \cdot (h \nabla u_{f_1})\|_{L^2}\|h\|_{L^2}\geq c'(g_{min}, K_{min}, B, \mathcal O,d)\|v\|^2_{L^2}\gtrsim \|h\|^2_{L^2},
		\end{equation*}
		which together with (\ref{diff-ub}) yields the desired estimate.
	\end{proof}
	
	\subsection{Schr\"odinger equation}\label{sec-schr-facts}
	
	\subsubsection{Estimates for $V_f$ and $G$}\label{sec-schr-G}
	In this section, for each $f\in C(\mathcal O)$ with $f\geq 0$, let $L_f$ denote the Schr\"odinger differential operator
	\begin{equation*}
	L_f:H^2_0(\mathcal O)\to L^2(\mathcal O),\qquad L_f[u]= \Delta u-2fu, 
	\end{equation*}
	where $H^2_0$ is given by (\ref{h0}). As in the divergence form case, $L_f$ is a bijection with a linear, continuous inverse operator which we again denote by
	\begin{equation*}
	V_f:L^2(\mathcal O)\to H^2_0(\mathcal O),\qquad \psi\mapsto V_f\left[\psi\right].
	\end{equation*}
	In other words, for any $f\in C(\mathcal O)$ and $\psi\in L^2$ the \emph{inhomogeneous} equation
	\begin{equation}\label{schr-2}
	\begin{cases}
	\Delta u-2fu =\psi \quad \textnormal{on} \quad \mathcal O,\\
	u=0\quad \textnormal{on} \quad \mathcal \partial O
	\end{cases}
	\end{equation}
	has a unique weak solution which we shall denote by $\omega_{f,\psi}:=V_f[\psi]\in H^2_0(\mathcal O)$, see Theorem 4 in Chapter 6.3 of \cite{evans} for this standard result for elliptic PDEs.
	\par
	As in the divergence form case, we first observe that for $p\in \{2,\infty\}$, the $L^p\to L^p$ operator norm of $V_f$ can be upper bounded uniformly in $f$.
	\begin{lem}\label{schafk}
		There exists a constant $C>0$ such that for all $f\in C(\mathcal O)$ with $f\geq 0$ and $\psi\in L^2(\mathcal O)$, $w_{f,\psi}=V_f[\psi]$ satisfies
		\[\|V_f\left[\psi\right]\|_{L^2}\leq C\|\psi\|_{L^2} \]
		and if $\psi\in C(\mathcal O)$, then also
		\[\|V_f\left[\psi\right]\|_{\infty}\leq C\|\psi\|_{\infty}. \]
	\end{lem}
	\begin{proof}
		We have the Feynman-Kac representation 
		$$w_{f,\psi}(x) = -\frac{1}{2}\mathbb E^x \Big[\int_0^{\tau_\mathcal O} \psi(X_s) e^{-\int_0^s f(X_r)dr}ds \Big], ~x \in \mathcal O,~ \psi \in C(\mathcal O),$$ where $(X_s: s \ge 0)$ is a standard $d$-dimensional Brownian motion started at $x$, with exit time $\tau_{\mathcal O}$  from $\mathcal O$, see p.84 and Theorem 3.22 of \cite{chungzhao} . [These results are applicable as $C(\mathcal O)\subseteq J$ with $J$ defined on p.62 of \cite{chungzhao}, and $C(\mathcal O)\subseteq \mathbb F(D,q)$ with $\mathbb F(D,q)$ defined on p.80 of \cite{chungzhao}.] The proof is now similar to that of Lemma \ref{div-lem-lp}, using $f \ge 0$ and that $\sup_{x\in\mathcal O}\E^x[\tau_{\mathcal O}]\leq K(vol(\mathcal O), d)<\infty$ by Theorem 1.17 in \cite{chungzhao}. 
	\end{proof}
	
	Using the above lemma, we  now show the following regularity estimate.
	
	\begin{lem}\label{lem-schr-h2}
		There exists a constant $C$ such that for all $f\in C^1(\mathcal O)$ with $f\geq 0$ and $\psi\in L^2(\mathcal O)$, we have
		\begin{align}
		\|V_f\left[\psi\right]\|_{H^2}&\leq C(1+\|f\|_{\infty})\|\psi\|_{L^2},	\label{schr-l2h2}\\
		\|V_f\left[\psi\right]\|_{L^2}&\leq C(1+\|f\|_{\infty})\|\psi\|_{(H^2_0)^*}.	\label{schr-h-2l2}
		\end{align}
	\end{lem}
	\begin{proof}
		By the norm equivalence (\ref{lapl-isom}) and (\ref{schr-2}), we have that
		\begin{equation*}
		\begin{split}
		\|V_f\left[\psi\right]\|_{H^2}&\lesssim \left\|\Delta V_f\left[\psi\right]\right\|_{L^2}\leq \|L_fV_f\left[\psi\right]\|_{L^2}+\|fV_f\left[\psi\right]\|_{L^2}\\
		&\leq \|\psi\|_{L^2}+\|f\|_{\infty}\|V_f\left[\psi\right]\|_{L^2}\lesssim (1+\|f\|_{\infty})\|\psi\|_{L^2},
		\end{split}
		\end{equation*}
		which proves (\ref{div-l2h2}). The second estimate (\ref{schr-h-2l2}) now follows from the same duality argument as in the proof of (\ref{div-h-2}).
	\end{proof}
	
	Next, we prove some basic boundedness properties of the forward map $G:f\mapsto u_f$.
	
	\begin{lem}\label{lem-schr-c2}
		Suppose that for some $g\in C^\infty(\partial \mathcal O)$, $\alpha>d/2$ and $K_{min}\geq 0$, $\tilde{\mathcal F}$ is as in (\ref{ftilde}), and let $u_f$ be the unique solution of (\ref{schroedeq}). 
		\par 
		1. There exists $C>0$ (independent of $g$) such that for all $f\in\tilde{\mathcal F}$, we have
		\[\|u_f\|_{\mathcal C^2(\mathcal O)}\leq C(1+\|f\|_{\infty})\|g\|_{\mathcal C^2(\mathcal O)}. \]
		\par
		2. There exists $C>0$ (possibly depending on $g$) such that for all $f\in\tilde{\mathcal F}$,
		\[\|u_f\|_{ H^{\alpha+2}(\mathcal O)}\leq C(1+\|f\|_{H^\alpha}^{\alpha/2+1}). \]
	\end{lem}
	\begin{proof}
		By (\ref{c-isom}), $(\Delta,\textnormal{tr}[\cdot])$ is a topological isomorphism between the spaces $\mathcal C^2(\mathcal O)$ and $\mathcal C^0(\mathcal O)\times \mathcal C^{2}(\partial O)$, whence we deduce that for all $u\in\mathcal C^2(\mathcal O)$, we have the norm estimate
		\[\|u\|_{\mathcal C^2(\mathcal O)}\leq C\left(\|\Delta u\|_{\mathcal C^0(\mathcal O)}+\|\textnormal{tr}[u]\|_{\mathcal C^2(\partial \mathcal O)} \right).  \]
		Using this, the PDE (\ref{schroedeq}) and the triangle inequality, we have for $f\in \mathcal F$,
		\begin{equation}\label{ufc2}
		\begin{split}
		\|u_f\|_{\mathcal C^2(\mathcal O)}&\lesssim \left\|L_fu_f \right\|_{\mathcal C^0(\mathcal O)}+\|fu_f\|_{\mathcal C^0(\mathcal O)}+\left\|\text{tr}[u_f]\right\|_{\mathcal C^2(\partial\mathcal O)}\\
		&\leq \|f\|_{\infty}\|u_f\|_{\infty}+\|g\|_{\mathcal C^2(\partial \mathcal O)}.
		\end{split}
		\end{equation}
		Next, we claim that there exists a constant $C>0$ such that for all $f,g$ as in the hypotheses, we have
		\begin{equation}\label{schr-hom-linf}
		\|u_f\|_{\infty}\leq C\|g\|_{\infty}.
		\end{equation}
		Indeed, this can be seen immediately from the fact that $f\geq 0$ and the Feynman-Kac representation (see \cite{chungzhao}, Theorem 4.7)
		\begin{equation}\label{schr-fk}
		u_f(x)=\frac{1}{2}\mathbb E^x\left[g(X_{\tau_{\mathcal O}})e^{-\int_0^{\tau_{\mathcal O}}f(X_s)ds}\right], \quad x\in\mathcal O, 
		\end{equation}
		where $(X_s:s\geq 0), \tau_\mathcal O$ are as in the proof of Lemma \ref{schafk}.  Hence, combining (\ref{schr-hom-linf}) with (\ref{ufc2}) yields the desired estimate
		\[\|u_f\|_{\mathcal C^2(\mathcal O)}\lesssim \|f\|_{\infty}\|g\|_{L^\infty(\mathcal O)}+\|g\|_{\mathcal C^2(\partial \mathcal O)}\leq (1+\|f\|_{\infty})\|g\|_{\mathcal C^2(\partial \mathcal O)}. \]
		\par 
		For the second part, we initially assume $f \in C^\infty(\mathcal O)$ so that $u_f \in C^\infty(\mathcal O)$ too (see Corollary 8.11 in \cite{gt}), and then use the topological isomorphism $(\Delta,\textnormal{tr})$ between $H^{\alpha+2}(\mathcal O)$ and $H^\alpha(\mathcal O)\times H^{\alpha+3/2}(\partial\mathcal O)$, which yields
		\begin{equation*}
		\begin{split}
		\|u_f\|_{H^{\alpha+2}(\mathcal O)}&\lesssim \|\Delta u_f\|_{H^\alpha(\mathcal O)}+\|\textnormal{tr}[u_f]\|_{H^{\alpha+3/2}({\partial\mathcal O})}\lesssim \|fu_f\|_{H^\alpha(\mathcal O)}+\|g\|_{C^{\alpha+2}(\partial \mathcal O)}\\
		&\lesssim 1+ \|f\|_{H^\alpha}\|u_f\|_{H^\alpha}\lesssim 1+\|u_f\|_{H^{\alpha+2}}^{\frac{\alpha}{\alpha+2}}\|u_f\|_{L^2}^{\frac{2}{\alpha+2}}\|f\|_{H^\alpha}.
		\end{split}
		\end{equation*}
		Dividing this by $\|u_f\|_{H^{\alpha+2}}^{\frac{\alpha}{\alpha+2}}$ when $\|u_f\|_{H^{\alpha+2}}\ge 1$ and otherwise estimating it by $1$, and using (\ref{schr-hom-linf}), we have that
		\[\|u_f\|_{H^{\alpha+2}}\lesssim 1+\|u_f\|_{L^2} \|f\|_{H^\alpha}^{\frac{\alpha+2}{2}} \lesssim  1+\|g\|_{\infty} \|f\|_{H^\alpha}^{\frac{\alpha+2}{2}}\lesssim1+\|f\|_{H^\alpha}^{\frac{\alpha+2}{2}}. \] The case of general $f \in \tilde {\mathcal F}$ now follows from taking smooth $f_n >K_{min}/2, f_n \to f $ in $H^\alpha$, showing that $u_{f_n}$ is Cauchy in $H^{\alpha+2}$ (by using (\ref{lapl-isom}), (\ref{h-int}), Lemma \ref{schafk}), and taking limits in the last inequality. Details are left to the reader.
	\end{proof}

	\subsubsection{Estimates for $G^{-1}$}
	
	\begin{lem}\label{lem-schr-stab} 
		Suppose that for some $\alpha>d/2,K_{min}\geq 0$, $g_{min}>0$ and $g\in C^\infty(\partial \mathcal O)$ with $\inf_{x \in \partial \mathcal O}g(x)\geq g_{min}$, $\tilde{\mathcal F}$ is given by (\ref{ftilde}), and let $u_f$ denote  the unique solution of (\ref{schroedeq}). Then there exist constants $c_1,c_2>0$ such that for all $f_1,f_2\in\tilde{\mathcal F}$, we have
		\begin{equation*}
		\begin{split}
		\|f_1-f_2\|_{L^2}\leq c_1\big( &e^{c_2\|f_1\|_{\infty}}\left\|u_{f_1}-u_{f_2}\right\|_{H^2}\\
		&\qquad +\|u_{f_2}\|_{C^2}\;e^{c_2\|f_1\vee f_2\|_{\infty}}\|u_{f_1}-u_{f_2}\|_{L^2}\big).
		\end{split}
		\end{equation*}
	\end{lem}
	
	\begin{proof}
		Applying Jensen's inequality to the Feynman-Kac representation (\ref{schr-fk}), and since $\sup_x \E^x\tau_\mathcal O \le c<\infty$ (see the proof of Lemma \ref{schafk}) yields
		\begin{equation}\label{schr-inf-est}
		\inf_{x\in\mathcal O}u_f(x)\geq g_{min}\inf_{x\in\mathcal O}e^{-\|f\|_{\infty}\E^x\tau_{\mathcal O}}\geq g_{min}e^{-c\|f\|_{\infty}}>0.
		\end{equation}
		Moreover, (\ref{schroedeq}) yields that we have $f=\frac{\Delta u_f}{2u_f}$ on $\mathcal O$, for all $f\in\tilde{\mathcal F}$. Thus, for any $f_1,f_2\in\tilde{\mathcal F}$, we can estimate
		\begin{equation}\label{schr-stab-1}
		\begin{split}
		\| f_1-&f_2\|_{L^2}=\frac 12\|\frac{\Delta u_{f_1}}{u_{f_1}}-\frac{\Delta u_{f_2}}{u_{f_2}} \|_{L^2}\\
		&\lesssim \left\|\left(\Delta u_{f_1}- \Delta u_{f_2}\right)u_{f_1}^{-1} \right\|_{L^2}+\left\|\Delta u_{f_2}\left(u_{f_1}^{-1}-u_{f_2}^{-1} \right)\right\|_{L^2}\\
		&\lesssim \big(\inf_{x\in\mathcal O}\big|u_{f_1}(x)\big|\big)^{-1}\left\|u_{f_1}-u_{f_2}\right\|_{H^2}+\big\|\Delta u_{f_2}\big\|_{C^2}\big\|u_{f_1}^{-1}-u_{f_2}^{-1}\big\|_{L^2}.
		\end{split}
		\end{equation}
		Further, using the mean value theorem and (\ref{schr-inf-est}), we have that
		\begin{equation*}
		\left|u_{f_1}^{-1}-u_{f_2}^{-1}\right|\leq \max\{u_{f_1}^{-2},u_{f_2}^{-2}\}\left|u_{f_1}-u_{f_2}\right|\leq g_{min}^{-2}e^{2c\|f_1\vee f_2\|_{\infty}}\left|u_{f_1}-u_{f_2}\right|.
		\end{equation*}
		Combining this with (\ref{schr-stab-1}) and using (\ref{schr-inf-est}) once more, we obtain that
		\[\left\| f_1-f_2\right\|_{L^2}\lesssim e^{c\|f_1\|_{\infty}}\left\|u_{f_1}-u_{f_2}\right\|_{H^2}+ e^{2c\|f_1\vee f_2\|_{\infty}}\left\|u_{f_1}-u_{f_2}\right\|_{L^2}, \]
		which concludes the proof.
	\end{proof}
	
	\section{Some properties of regular link functions}\label{sec-reg}
	We define $L^p$-norms, $0<p \le \infty$, in the usual way. By obvious modifications, the following lemma holds also for regular functions $\Phi:(a,b)\to \R$ with arbitrary $-\infty\leq a<b\leq \infty$ and suitable $F,J:\mathcal O\to (a,b)$, we restrict to the case $(a,b)=\R$ here.
	
	\begin{lem}\label{lem-phi}
		Suppose $\Phi:\R \to\R$ is a smooth and regular function in the sense of (\ref{phibddder}).
		\par 1. There exists $C<\infty$ such that for all $p\in[1,\infty]$,
		\begin{equation}\label{phi-lp}
		\forall F\in L^p(\mathcal O), ~~\|\Phi\circ F\|_{L^p}\leq C(1+\|F\|_{L^p}).
		\end{equation}
		\par 
		2. For each integer $m\geq 0$, there exists $C<\infty$ such that 
		\begin{equation}\label{phi-cm}
		\forall F\in C^m(\mathcal O),~~\|\Phi\circ F\|_{C^m}\leq C\left(1+\|F\|_{C^m}^m\right).
		\end{equation}
		\par 
		3. For each integer $m\geq d/2$, there exists $C<\infty$ such that for all $F\in H^m(\mathcal O)$, we have $\Phi\circ F\in H^m(\mathcal O)$ and
		\begin{equation}\label{phi-sob}
		\|\Phi\circ F\|_{H^m}\leq C(1+\|F\|_{H^m}^m).
		\end{equation}
		\par 
		4. There exists $C<\infty$ such that for $\kappa\in\{1,2\}$ and all $F, J\in C^\kappa(\mathcal O)$,
		\begin{equation}\label{phi-neg-sob}
		\left\|\Phi\circ F-\Phi\circ J\right\|_{(H^{\kappa})^*}\leq C\left\|F-J\right\|_{(H^{\kappa})^*}\left(1+ \|F\|_{C^\kappa}^\kappa\vee \|J\|_{C^\kappa}^\kappa\right).	
		\end{equation}
		
	\end{lem}
	
	The rest of this section is devoted to proving Lemma \ref{lem-phi}. To prove (\ref{phi-cm})-(\ref{phi-sob}), we need Fa\'a di Bruno's formula (a generalization of the chain rule), which classically holds for $C^m$ functions, and by the chain rule for Sobolev functions (see e.g. \cite{ziemer}, Thm. 2.1.11) also holds for $H^m$ functions.
	\begin{lem}
		Let $m\in\N$ and suppose that $F:\mathcal O\to \mathbb R$ and $\Phi:\mathbb R\to \mathbb R$ are of class $H^m(\mathcal O)$ and $C^m(\R)$ respectively. Then for any $\alpha\in\{1,...,d\}^{m}$, the $m$-th order partial derivative of $f:=\Phi\circ F$ in direction $x_{\alpha_1}... x_{\alpha_m}$ is given by
		\begin{equation}\label{faadibruno}
		\frac{\partial^m f}{\partial x_{\alpha_1}...\partial x_{\alpha_m}}(x)=\sum_{\pi\in \Pi}\Phi^{(|\pi|)}(F(x))\prod_{B\in\pi}\frac{\partial ^{|B|}F}{\prod_{j\in B} \partial x_{\alpha_j}}(x),
		\end{equation}
		where $\pi$ runs through the set $\Pi$ of all partitions of $\{1,...,m\}$, and the $B\in\pi$ runs over all `blocks' $B$ of each partition $\pi$.
	\end{lem}

	\begin{proof}[Proof of (\ref{phi-lp})-(\ref{phi-cm})]
		By (\ref{phibddder}), there exists a constant $c>0$ only depending on the values of $\Phi(0)$ and $\|\Phi'\|_{\infty}$ such that for all $x\in\R$, $|\Phi(x)|\leq c(1+|x|)$, which yields (\ref{phi-lp}). For (\ref{phi-cm}), let $\alpha\in\{1,...,d\}^{|\alpha|}$, $1\leq |\alpha|\leq m$ and let $\pi$ be a partition of $\{1,...,|\alpha|\}$. Then the corresponding summand on the right side of (\ref{faadibruno}) can be estimated by 
		\[ \left\| \prod_{B\in\pi}\frac{\partial ^{|B|}F}{\prod_{j\in B} \partial x_{\alpha_j}}\right\|_{\infty} \leq \|F\|_{C^m}^{|\pi|}\lesssim \left(1+\|F\|_{C^m}^{m}\right). \]
		By summing the above display over all such $\alpha,\pi$ and using (\ref{phi-lp}) with $p=\infty$, we obtain (\ref{phi-cm}).
	\end{proof}
	
	To prove (\ref{phi-sob}), we also need the Gagliardo-Nirenberg interpolation inequality (see \cite{nirenberg59}, p.125) in the special case $r=q=2$.
	\begin{lem}\label{lem-gns}
		Suppose that $\mathcal O\subseteq \mathbb R^d$ is a bounded $C^\infty$ domain and that $i=1,...,m$, $a\in [i/m,1]$ and $p\in [1,\infty)$ satisfy
		\begin{equation}\label{gns-cond}
		\frac 1p=\frac 12 +\frac i d - \frac md a.
		\end{equation}
		
		Then for any $s>0$, there exist constants $C_1,C_2$ depending only on $m, d, i, a, \mathcal O$ and $s$ such that for all $F \in H^m$, we have that $D^iF\in L^p$, and 
		$$\|D^iF\|_{L^p}\leq C_1\|D^mF\|_{L^2}^a\|F\|_{L^2}^{1-a}+C_2\|F\|_{L^s}.$$
	\end{lem}

	\begin{proof}[Proof of (\ref{phi-sob})]
		Let us write $f=\Phi\circ F$. By (\ref{phi-lp}), we have that $\|f\|_{L^2}\leq C(1+\|F\|_{L^2})$ whence we only need to estimate $\|D^mf\|_{L^2}$. For any $\alpha\in\{1,...,d\}^m$ we have by $(\ref{faadibruno})$ that
		\begin{equation*}
		\begin{split}
		\left|\frac{\partial^m f}{\partial x_{\alpha_1}...\partial x_{\alpha_m}}(x)\right|^2\lesssim \sum_{\pi\in\Pi}\left| \prod_{B\in\pi}\frac{\partial ^{|B|}F}{\prod_{j\in B} \partial x_{\alpha_j}}(x)\right|^2
		\end{split}
		\end{equation*}
		Similarly to the proof of (\ref{phi-cm}), it thus suffices to prove that for all $\alpha\in\{1,...,d\}^m$ and partition $\pi$ of $\{1,...,m\}$,
		\begin{equation}\label{hm-red}
		\left\|\prod_{B\in\pi}\frac{\partial ^{|B|}F}{\prod_{j\in B} \partial x_{\alpha_j}} \right\|_{L^2}\lesssim(1+\|F\|_{H^m}^m).
		\end{equation}
		Fix some $\pi$ for the rest of the proof. For $i=1,...,m$, define
		$$\pi_i:=\left\{B\in\pi\;\middle| \;|B|=i\right\}, \qquad p_i:=\frac{2m}{i}.$$
		Then we have $\sum_{i=1}^{m}i|\pi_i| = m$, and hence by H\"older's inequality
		\begin{equation}\label{holder}
		\begin{split}
		\left\|\prod_{B\in\pi}\frac{\partial ^{|B|}F}{\prod_{j\in B} \partial x_{\alpha_j}} \right\|_{L^2}&\leq \left\|\prod_{i=1}^m|D^iF|^{|\pi_i|} \right\|_{L^2}\leq \prod_{i=1}^m\left\||D^iF|^{|\pi_i|}\right\|_{L^{p_i/|\pi_i|}}\\
		&=\prod_{i=1}^m\left\|D^iF\right\|_{L^{p_i}}^{|\pi_i|}.
		\end{split}
		\end{equation}
		
		Next, define 
		\begin{equation}\label{ai}
		a_i:=\left(\frac id+\frac 12-\frac i{2m}\right)\frac dm\qquad  \textnormal{for } i=1,...,m.
		\end{equation}
		To apply Lemma \ref{lem-gns}, we verify that for each $i=1,...,m$, $(i,a_i,p_i)$ satisfies the conditions of Lemma \ref{lem-gns}. By definition, (\ref{gns-cond}) is satisfied. Moreover, as $i\leq m$, it follows that
		$$ma_i=i+\left(\frac d2-\frac{di}{2m}\right)\geq i,$$
		whence we have $\frac{i}{m}\leq a_i$. Finally, we need to verify $a_i\leq 1$. For this, we note that for $i=1,...,m$, choosing $m=d/2$ in (\ref{ai}) yields $a_i=a_i(m)=1$. Moreover, for $m\geq d/2$, we have
		\begin{equation}
		\begin{split}
		\frac{\partial a_i(m)}{\partial m}=\frac{2di-2mi-dm}{2m^3}\leq \frac{di-dm}{2m^3}\leq 0,
		\end{split}
		\end{equation}
		so that $\alpha_i\leq 1$.
		\par
		Applying Lemma \ref{lem-gns} with $s=2$ to (\ref{holder})  and using that $\sum_{i=1}^m|\pi_i|\in [1,m]$ yields that
		\begin{align*}
		\left\|\prod_{B\in\pi}\frac{\partial ^{|B|}F}{\prod_{j\in B} \partial x_{\alpha_j}} \right\|_{L^2}&\lesssim
		\prod_{i=1}^m \left(\|D^mF\|_{L^2}^{a_i}\|F\|_{L^2}^{1-a_i}+\|F\|_{L^2}\right)^{|\pi_i|}\\
		&\lesssim \prod_{i=1}^m \|F\|_{H^m}^{|\pi_i|}\lesssim 1+\|F\|_{H^m}^m.
		\end{align*}
	\end{proof}
	\begin{proof}[Proof of (\ref{phi-neg-sob})]
		1. Let $\kappa\in \{1,2\}$ and fix $F,J\in C^\kappa(\mathcal O)$. Define the function
		\[\omega:\mathcal O\to \mathbb R,\qquad \omega(x):=
		\begin{cases}
		\frac{\Phi(F(x))-\Phi(J(x))}{F(x)-J(x)} \quad \text{if} \;\; x\in\{F\neq J\}\;\\
		\Phi'(F(x))\quad \text{if} \;\; x\in \{F= J\}.
		\end{cases} \]
		Then we have, using also (\ref{c-h-mult}), that
		\begin{equation*}
		\begin{split}
		\|\Phi\circ F-\Phi\circ J\|_{(H^{\kappa})^*}&=\sup_{\varphi\in C^\infty(\mathcal O), \; \|\varphi\|_{H^\kappa}\leq 1}\left|\int_{\mathcal O} \varphi(\Phi\circ F-\Phi\circ J)\mathbbm 1_{\{F\neq J\}} \right|\\
		&=\sup_{\varphi\in C^\infty(\mathcal O), \; \|\varphi\|_{H^\kappa}\leq 1}\left|\int_{\mathcal O} (F-J) \varphi  \omega \right|\\
		&\leq \|F-J\|_{(H^{\kappa})^*}\sup_{\varphi\in C^\infty(\mathcal O), \; \|\varphi\|_{H^\kappa}\leq 1}\left\|\varphi \omega\right\|_{H^k}\\
		&\lesssim  \|F-J\|_{(H^{\kappa})^*}\left\|\omega\right\|_{C^\kappa}.
		\end{split}
		\end{equation*}
		\par
		2. Thus it suffices to prove that $\left\|\omega\right\|_{C^\kappa}\leq C (1 + \|F\|_{C^\kappa}^\kappa\vee \|J\|_{C^\kappa}^\kappa)$ for some $C>0$ independent of $F$ and $J$. Writing $\omega=\psi\circ \phi$, where
		\[\phi:\mathcal O\to \mathbb R^2, \qquad \phi(z)=\left(F(z),J(z)\right),\]
		\[\psi:\mathbb R^2\to (0,\infty), \qquad \psi(x,y)=
		\begin{cases}
		\frac{\Phi(x)-\Phi(y)}{x-y} \quad \text{if} \;\; x\neq y\\
		\Phi'(x)\quad \text{if} \;\; x=y, 
		\end{cases}\]
		we see by the multivariate chain rule that it suffices to show that $\psi$ is $\kappa$-times continuously differentiable with bounded derivatives, and we achieve this by showing that the partial derivatives of $\psi$ of order $\kappa$ exist and are continuous throughout $\mathbb R^2$.
		\par
		3. We will repeatedly use the following basic fact: Let $h:\mathbb R\to \mathbb R$ be continuous and continuously differentiable on $\mathbb R\setminus \{0\}$. If $h'$ has a continuous extension $g$ to $\mathbb R$ with some value $g(0)=\xi$, then $h\in C^1(\mathbb R)$ with $h'(0)=\xi$.
		\par 
		4.  Clearly, $\psi$ is smooth on $\mathbb R^2\setminus\{x=y\}$. For $k\geq 0$  and $x,y\in \mathbb R$, we denote the remainder of the $k$-th order Taylor expansion by
		$$R_{k,x}(y):=\Phi(y)-\sum_{j=0}^{k}\frac{\Phi^{(j)}(x)}{j!}(y-x)^j.$$
		For $x\neq y$, we have $\psi(x,y)=\frac{R_{0,x}(y)}{y-x}$ and also, by induction
		\begin{equation}\label{partder}
		\partial_1^k\psi(x,y)=\frac{k!R_{k,x}(y)}{(y-x)^{k+1}}, \quad k\geq 0,
		\end{equation}
		where $\partial_1$ denotes the partial derivative with respect to $x$. By the mean value form of the remainder, we know that $R_{k,x}(y)=\frac{\Phi^{(k+1)}(\xi)}{(k+1)!}(y-x)^{k+1}$ for some $\xi$ between $x$ and $y$. Thus we can continuously extend $\partial_1^k\psi$ to $\{x=y\}$ by
		$$\partial_1^k\psi(x,x)=\frac{\Phi^{(k+1)}(x)}{k+1}.$$
		It follows that the partial derivatives with respect to $x$ of all orders exist and are continuous on $\mathbb R^2$. The same holds for the partial derivatives with respect to $y$, by symmetry, concluding the proof of the case $\kappa=1$. The case $\kappa=2$ follows by adapting the previous arguments for mixed partial derivative $\partial_1\partial_2$ and is left to the reader.
	\end{proof}

	\section{Proof of Theorem \ref{thm-gen}, Part 1} \label{sec-ex-pf}
	
	Let $\lambda,\eps>0$ be fixed throughout and let us write $\mathscr J=\mathscr J_{\lambda, \eps}$. We denote by $\mathcal T_w=\mathcal T_{w,\alpha}$ the weak topology on $\mathcal H$ (recall $\mathcal H=H^\alpha(\mathcal O)$ if $\kappa<1/2$ and $\mathcal H=H^\alpha_c(\mathcal O)$ if $\kappa\geq 1/2$), i.e. the coarsest topology with respect to which all bounded linear functionals $L:\mathcal H\to \mathbb R$ are continuous. We also denote the subspace topology on subsets of $\mathcal H$ by $\mathcal T_w$. On any closed ball $\mathcal H(R):=\{F\in \mathcal H:\|F\|_{H^\alpha} \le R\}$, this topology is metrisable by some metric $d$, see e.g. Theorem 2.6.23 in \cite{megg}.

	\paragraph{Step 1: Localisation} In Lemma \ref{lem-main}, by assumption on $\alpha$, we have that $\Psi_*(\lambda,R)/R^2\xrightarrow{R\to\infty}0$ and so there exists $\delta>0$ such that for all $R\geq \delta$, we have that $R^2\geq c_1\eps\Psi_*(\lambda,R)$, where $c_1$ is the constant from (\ref{delta-req}). Thus, applying Theorem \ref{thm-ep}, we have that the events
	\[A_j:=\left\{\mathscr J\text{ has a maximizer } \hat F\notin \mathcal V\cap \mathcal H(2^j) \right\}\]
	satisfy $\P(A_j)\xrightarrow{j\to\infty} 0$, whence choosing $j\in\N$ large enough ensures that
	$$\sup_{F\in\mathcal V\cap H^{\alpha}(2^j)}\mathscr J(F)=\sup_{F\in\mathcal V}\mathscr J(F)$$ 
	holds with probability as close to one as desired.
	
	\paragraph{Step 2: Local existence via direct method}
	By the previous step, it suffices to show that for any $j\in\N$, $\mathscr J$ almost surely has a maximizer over $\mathcal V\cap \mathcal H(2^j)$. We fix some $j\in\N$. As $\mathcal V$ is weakly closed and $\mathcal H(2^j)$ is weakly sequentially compact by the Banach-Alaoglu Theorem, it follows that any sequence $F_n\in \mathcal V\cap \mathcal H(2^j)$ has a weakly convergent subsequence $F_n\to F$ with weak limit $F\in\mathcal V\cap \mathcal H(2^j)$. Moreover, we claim that $-\mathscr J:\mathcal V\cap \mathcal H(2^j)\to \R$ is lower semicontinuous with respect to $\mathcal T_w$. To see this, we decompose $-\mathscr J$ as
	\[-\mathscr J(F)=-2\langle Y,\mathscr G(F)\rangle_\mathbb H+\|\mathscr G(F)\|_{\mathbb H}^2+\lambda^2\|F\|_{H^\alpha}^2=:I+II+III.\]
	The term $I$ is, almost surely under $\mathbb P_{F_0}^\eps$, continuous w.r.t. $\mathcal T_w$ by Lemma \ref{lem-Ycont}, $II$ is continuous w.r.t. $\mathcal T_w$ by Lemma \ref{lem-IIcont} and $III$ is lower semicontinuous by a standard fact from functional analysis. Thus the existence of minimisers follows from the direct method of the calculus of variations. %\end{proof}

	The next three lemmas are needed to prove lower semicontinuity of $-\mathscr J$.
	
	\begin{lem}\label{lem-weakcon}
		Let $\alpha>0$ and let $\left(F_n: n\in\mathbb N\right)\subseteq \mathcal H$, for $\mathcal H=H^\alpha$ or $H^\alpha_c$, be a sequence such that $F_n\to F$ for $\mathcal T_w$. Then also $F_n\to F$ in $L^2$.
	\end{lem}
	\begin{proof}
		It suffices to show that for any subsequence $(F_{n_j}:j\in\mathbb N)$, there exists a further subsequence $(F_{n_{j'}}:j'\in\mathbb N)$ such that $F_{n_{j'}}\to F$  in $L^2$. By the uniform boundedness principle, there exists $R>0$ such that for all $n\in\mathbb N$, $\|F_n\|_{H^\alpha}\leq R$. By the Rellich-Kondrashov compactness theorem, the closed ball $\mathcal H(R)$ is pre-compact with respect to $L^2$ topology, hence for any subsequence $(F_{n_j})$ of $(F_n)$, there exists a further convergent subsequence $(F_{n_{j'}})$ with limit $\tilde F$ in $L^2$. In particular, we have $F_n\to F$ weakly in $L^2$ and $F_{n_{j'}}\to \tilde F$ in $L^2$, so that by the uniqueness of weak limits, we have $\tilde F=F$ as elements in $L^2$, and therefore $\tilde F=F$ a.e. in $\mathcal O$ and $F_n\to F$ in $L^2$.
	\end{proof}
	
	\begin{lem}\label{lem-IIcont}
		Let $\alpha>0$, $\kappa, \gamma\in \R_+$ and $\mathcal V_0\subseteq \mathcal V$ be a bounded subset of $\mathcal H=H^\alpha$ or $H^\alpha_c$. If a map $\mathscr G: \mathcal V\to \mathbb H$ is $(\kappa,\gamma,\alpha)$-regular, then it is continuous as a mapping from $(\mathcal V_0,d)$ to $\mathbb H$.
	\end{lem}
	\begin{proof}
		Take any $F_n,F\in \mathcal V_0$ such that $F_n\to F$ for $\mathcal T_w$ and note that $\|F_n\|_{H^\alpha} \le R$ for some $R>0$. By Lemma \ref{lem-weakcon} we have $\|F_n - F\|_{L^2}\to 0$ and by (\ref{entrcond}) and the continuous imbedding $L^2 \subseteq (H^\kappa)^*, \kappa \ge 0$, we obtain
		\begin{equation}
		\|\mathscr G(F_n)-\mathscr G(F)\|_{\mathbb H}\leq C\left(1+R^\gamma\right) \|F_n-F\|_{L^2} \xrightarrow{n\to\infty} 0.
		\end{equation}
	\end{proof}

	We finally establish a continuity result for the Gaussian process $Y^{(\eps)}$. 
	
	\begin{lem}\label{lem-Ycont}
		Suppose that $Y^{(\eps)}$ and $\mathscr G$ are as in Theorem \ref{thm-gen}. Then there exists a version of the Gaussian white noise process $\W$ in $\mathbb H$ such that for all $R>0$, the map (between metric spaces)
		\[ \Psi: (\mathcal V\cap \mathcal H(R),d)\to \mathbb R, \qquad F\mapsto \langle Y^{(\eps)}, \mathscr G(F)\rangle_{\mathbb H}\]
		is almost surely uniformly continuous.
	\end{lem}
	\begin{proof}
		For any $\delta>0$, define the modulus of continuity  	\[M_\delta :=\sup_{F,H\in \mathcal V\cap \mathcal H(R),\;\; d(F,H)\leq \delta}\left|\langle Y^{(\eps)},\mathscr G(F)-\mathscr G(H)\rangle_\mathbb H \right|,\] a random variable.
		Moreover, we define the set
		\[A:=\left\{\omega\in\Omega\;\middle|\; M_\delta\xrightarrow{\delta\to 0}0 \right\},\] where $\Omega$ is a probability space supporting the law $\mathbb P$ of $\mathbb W$. It is sufficient to show that $\mathbb P(A)=1$, and noting that $M_\delta$ is decreasing in $\delta$, it hence suffices to prove $\mathbb E\left[M_\delta \right]\xrightarrow{\delta\to 0}0$. To see this, similarly to the proof of Lemma \ref{thm-ep}, we apply Dudley's theorem (see \cite{nicklgine}, Theorem 2.3.7) to the Gaussian process
		\[\left(\mathcal \W(\psi):\;\psi\in \mathcal D_R\right), \qquad \mathcal D_R:=\left\{\mathscr G(F)\;\middle|\; F\in \mathcal V\cap \mathcal H(R)\right\}. \]
		For any $\delta>0$, define
		\[R_\delta:=\sup_{F,H\in \mathcal V\cap \mathcal H(R),\;\; d(F,H)\leq \delta}\|\mathscr G(F)-\mathscr G(H)\|_{\mathbb H}.\]
		By Lemma \ref{lem-IIcont}, we know that $\mathscr G$ is continuous as a mapping from $(\mathcal V\cap \mathcal H(R), d)$ to $\mathbb H$. As $(\mathcal V\cap \mathcal H(R), d)$ is a compact metric space, $\mathscr G$ is in fact uniformly continuous, so we have that $R_\delta\xrightarrow{\delta\to 0} 0$. By the same argument as in the proof of Lemma \ref{lem-main} (but choosing here $m:=(1+R^\gamma)$) we can use (\ref{entrcond}) to obtain
		\[H(\rho, \mathcal D_R, \|\cdot\|_{\mathbb H})\lesssim \left(\frac{Rm}{\rho}\right)^{\frac{d}{(\alpha+\kappa)}}, ~~~\rho>0, \]
		whence by Dudley's theorem, the modulus of continuity is controlled by
		\begin{equation*}
		\begin{split}
		\mathbb E\left[M_\delta\right]&\leq \mathbb E\left[\sup_{\psi,\varphi\in \mathcal D_R,\; \|\psi-\varphi\|_{\mathbb H}\leq R_\delta}\left|\langle\W,\psi-\varphi\rangle_\mathbb H\right|\right] \lesssim \int_0^{R_\delta}\left(\frac{Rm}{\rho}\right)^{\frac{d}{2(\alpha+\kappa)} }d\rho,
		\end{split}
		\end{equation*}
		which converges to zero as $\delta \to 0$ since $\alpha>d/2-\kappa$.
	\end{proof}

\section*{Acknowledgments}
The authors are grateful to two anonymous referees and an associate editor for their remarks and suggestions. RN thanks Francois Monard and Gabriel P. Paternain for helpful discussions.

\bibliography{plsreferences}
\bibliographystyle{siamplain}

\end{document}